\documentclass[sn-mathphys]{sn-jnl}

\usepackage[caption=false]{subfig}

\jyear{2021}%

\theoremstyle{thmstylethree}
\newtheorem{theorem}{Theorem}[section]
\newtheorem{lemma}{Lemma}[section]

\begin{document}

\title[Sufficient recovery conditions for noise-buried low rank tensors]{Sufficient recovery conditions for noise-buried low rank tensors}


\author*[1]{\fnm{Sergey} \sur{Petrov}}\email{spetrov.msk@gmail.com}

\author[1]{\fnm{Nikolai} \sur{Zamarashkin}}\email{nikolai.zamarashkin@gmail.com}

\affil*[1]{\orgname{INM RAS}, \orgaddress{\street{Gubkina, 8}, \city{Moscow}, \postcode{119333}, \country{Russia}}}


\abstract{ Low-rank tensor approximation error bounds are proposed for the case of noisy input data that depend on
	low-rank representation type, rank and the dimensionality of the tensor. The bounds show that high-dimensional low-rank structured approximations provide superior noise-filtering properties compared to matrices with the same rank and total element count.
	}




\maketitle

\section{Introduction}\label{sec:sec1}
Presence of additive noise is an inevitable practical aspect of many real-world problems that involve physical measurements, including, for example, medical image processing \cite{medic} and channel estimation in wireless communications \cite{worst}. A common basis of noise filtering algorithms is the utilization of the intrinsic structure of the underlying raw data. 

A possible variant of such a structure is sparsity: sparse linear system solutions can be less affected by a right hand side perturbation than dense solutions \cite{sparse}. Another frequent structural occurrence is a low-rank data matrix: theoretical bounds suggest that the eigen subspace \cite{kahan} and the singular subspace \cite{wedin} perturbations caused by additive noise depend on a spectral norm of the additive noise, which is, in most cases, much smaller than the Frobenius norm. That makes low-rank projections a viable tool for denoising, which is often employed for image and video processing \cite{image}\cite{video}.

The aim of this work is to generalize the noisy low-rank approximation studies towards higher-dimensional tensor low-rank structures. Intuitively, the structural denoising properties are expected to improve with the reduction of the number of the structure parameters, since more averaging takes place. The validity of such a phenomena in practice is established in several papers concerning wireless channel estimation: it is seen that tensor low-rank decompositions provide better estimation quality as compared to matrix decompositions applied to the same data \cite{wireless1}\cite{wireless2}.

Our work is focused on establishing a theoretical support for the phenomenon: an approximation quality bound that decays with dimensionality. A noticeable obstacle for low-rank tensor approximation bounds is the fact that the set of low-rank tensors may not be closed, and thus {\it optimal} low-rank approximation algorithms are practically unavailable. To deal with that, we are going to introduce an alternate {\it sufficient approximation condition}, and utilize the concept of 'Gaussian width' \cite{widthbook}, that allows to characterize the denoising properties of almost-arbitrary {\it bounded} sets. We will not be able to use the known Gaussian width properties directly, however, due to a tensor-specific challenge: that is, the set of rank-$R$ canonical tensors of dimensionality $d \geq 3$ with a bounded norm {\it cannot} be represented as a Minkowski sum of $R$ sets of rank-one tensors with a bounded norm. A significant technical part of the paper will be devoted to overcoming that challenge in the particular case of $R = 2$.
 
Using the introduced sufficient approximation algorithm condition, which is commonly fulfilled for the well-known algorithms like ALS \cite{als}, HOSVD \cite{hosvd}, TT-SVD \cite{tt} in practice, and the modified 'Gaussian width' concept, a theoretical approximation bound will be established that is closely fulfilled in practical experiments on synthetic data subject to white noise.

\section{Noise perturbation functional definition}
\begin{figure} 
	\begin{center}
	\includegraphics[width = 1.0\linewidth]{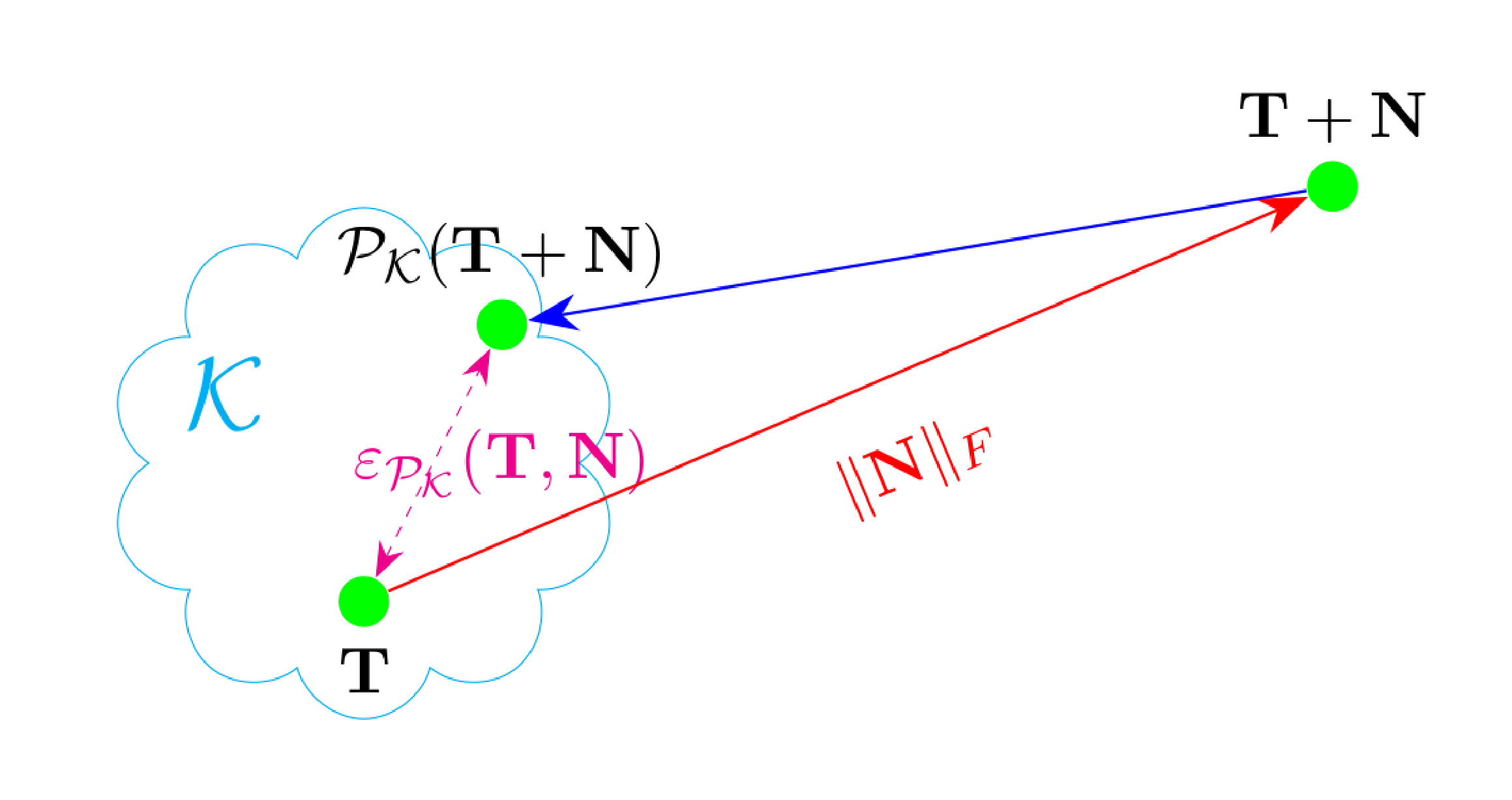}
	\end{center}
	\caption{Scheme of the proposed analysis}
	\label{fig:scheme}
\end{figure}

Assume ${\cal K}$ denotes the set of low-rank $d$-dimensional tensors 
\begin{equation}
{\cal K} = \{ {\bf T} \in \mathbb{R}^{m_1 \times m_2 \times \dots \times m_d}: \mbox{rank}({\bf T}) \leq R \}. \label{eq:set}
\end{equation}
The term $\mbox{rank}({\bf T})$ used above should correspond to one of the three low-rank tensor formats: canonical format, Tucker format or the Tensor-Train \cite{tt} format. Note that in all three cases, the set ${\cal K}$ satisfies the cone condition: 
\begin{equation}
\alpha \in \mathbb{R}, {\bf T} \in {\cal K} \Rightarrow \alpha {\bf T} \in {\cal K}. \label{eq:cone}
\end{equation}
Fix a low-rank tensor ${\bf T} \in {\cal K}$ and consider its perturbation ${\bf T + N}$, where each element of the noise tensor ${\bf N} \in \mathbb{R}^{m_1 \times m_2 \times \dots \times m_d}$ is an i.i.d. Gaussian random variable with mean $0$ and a predefined variance.

Consider an operator ${\cal P_K}: \mathbb{R}^{m_1 \times m_2 \times \dots \times m_d} \rightarrow {\cal K}$ with action defined by an approximation algorithm: in particular, ${\cal P_K}({\bf T + N})$ might denote the result of ALS \cite{als}, HOSVD \cite{hosvd} or TT-SVD \cite{tt} algorithm applied to ${\bf T + N}$. The main focus of this study is defined by a quantity
\begin{equation}
\label{eq:objective}
\varepsilon_{{\cal P_K}}({\bf T}, {\bf N}) := \| {\cal P_K}({\bf T} + {\bf N}) - {\bf T} \|_F.
\end{equation}
Figure \ref{fig:scheme} provides a graphical representation of the considered phenomena: empirical experiments show that commonly $\varepsilon_{{\cal P_K}}({\bf T}, {\bf N}) \ll \| {\bf N} \|_F$. Denote $M := \prod_{s = 1}^d m_s$ as the total number of tensor elements. The goal of this research is to provide a sharp theoretical bound for $\varepsilon_{{\cal P_K}}({\bf T}, {\bf N})$ that decreases with the growth of $d$ for fixed $M$. That would prove that high-dimensional low-rank tensor formats offer better de-noising properties compared to matrices.

\section{Gaussian width framework}

In order to establish a bound, we are going to utilize a framework that is commonly associated (but not fully equivalent to) with the concept of 'Gaussian width' \cite{widthbook}\cite{width}. The framework has minimal requirements for the considered set ${\cal K}$ in the form of the cone condition (\ref{eq:cone}). 

So far, no assumptions have been made about the approximation algorithm ${\cal P_K}$. Let us introduce the following hypothesis to express the algorithm approximation properties:
\begin{equation}
\label{eq:hypo}
\| {\cal P}_K({\bf \tilde{T}}) - ({\bf \tilde{T}}) \|_F \leq \| {\bf T} - {\bf \tilde{T}} \|_F = \| {\bf N} \|_F,
\end{equation}
where ${\bf \tilde{T} := T + H}$ is used for shortened notation. The proposed assumption means that ${\cal P}_K({\bf \tilde{T}}) \in {\cal K}$ is a better approximation to ${\bf \tilde{T}}$ than one particular element of ${\cal K}$. Thus, the hypothesis always holds true if the algorithm ${\cal P}_K$ returns an {\it optimal} approximation, but can still be considered even if ${\cal P}_K$ is suboptimal - which is true for most practically used tensor low-rank approximation algorithms. On the scheme on Figure \ref{fig:scheme}, the hypothesis (\ref{eq:hypo}) holds when the blue arrow is shorter than the right one.

Considering an arithmetic equality
\begin{eqnarray*}
	\| {\cal P}_{\cal K} ({\bf \tilde{T}}) - {\bf \tilde{T}} \|_F^2 
	& = & \| ({\cal P}_{\cal K}({\bf \tilde{T}}) - {\bf T}) - {\bf N} \|_F^2 \\
	& = & \| {\cal P}_{\cal K}({\bf \tilde{T}}) - {\bf T} \|_F^2 + \| {\bf N} \|_F^2 - 2 ({\cal P}_{\cal K}({\bf \tilde{T}}) - {\bf T}, {\bf N})_F,
\end{eqnarray*}
and applying the hypothesis (\ref{eq:hypo}), we can obtain
\begin{equation*}
\| {\cal P}_{\cal K}({\bf \tilde{T}}) - {\bf T} \|_F^2  \leq 2 \left({\cal P}_{\cal K}({\bf \tilde{T}}) - {\bf T}, {\bf N} \right)_F,
\end{equation*}
which, scaled by $\| {\cal P}_{\cal K}({\bf \tilde{T}}) - {\bf T} \|_F,$ gives  
\begin{equation*} 
\varepsilon_{{\cal P}_{\cal K}}({\bf T}, {\bf N}) = \| {\cal P}_{\cal K}({\bf \tilde{T}}) - {\bf T} \|_F 
\leq 2\left(\frac{{\cal P}_{\cal K}({\bf \tilde{T}}) - {\bf T}}{\| {\cal P}_{\cal K}({\bf \tilde{T}}) - {\bf T} \|_F}, {\bf N}\right)_F.
\end{equation*}
Let us denote the set 'self-subtraction' operation by 
\begin{equation*}
{\cal K} - {\cal K} = \{ {\bf A} - {\bf B}: {\bf A},{\bf B} \in {\cal K} \}.
\end{equation*}
Denoting ${\bf Z} = \frac{{\cal P}_{\cal K}({\bf \tilde{T}}) - {\bf T}}{\| {\cal P}_{\cal K}({\bf \tilde{T}}) - {\bf T} \|_F}$ and considering that $\bf Z$ satisfies ${\bf Z} \in {\cal K} - {\cal K}$, $\| {\bf Z} \|_F = 1$, we can bound 
$\varepsilon_{{\cal P}_{\cal K}}$ by 
\begin{equation} \label{eq:gauss}
\varepsilon_{{\cal P}_{\cal K}}({\bf T}, {\bf N}) \leq 2 ({\bf N}, {\bf Z})_F \leq 2 
\sup\limits_{ {\bf V} \in {\cal K} - {\cal K}, \\  \|{\bf V}\|_F = 1 } ({\bf N}, {\bf V})_F.
\end{equation}
Let us introduce a notation for the right side of~(\ref{eq:gauss}):
\begin{equation}
\label{eq:knorm}
\| {\bf N} \|_{\cal K} := \sup \limits_{{\bf V} \in {\cal K} - {\cal K}, \|{\bf V}\|_F = 1} ({\bf N}, {\bf V})_F.
\end{equation}
Note that $\| {\bf N} \|_{\cal K} \leq \| {\bf N} \|_F$, because 
$\| {\bf N} \|_F = \sup \limits_{\|{\bf V}\|_F = 1} ({\bf N}, {\bf V})$, 
where the optimization is done over a larger set compared to that in the definition of $\| {\bf N} \|_{\cal K}$. While positive-definiteness of $\| {\bf N} \|_{\cal K}$ may not hold for general ${\cal K}$, it can be seen that for the particular choice of ${\cal K}$ (\ref{eq:set}) the value $\| {\bf N} \|_{\cal K}$ is an actual norm. In case when ${\bf N}$ consists of random i.i.d. Gaussian variables, the value $\mathbb{E} \| {\bf N} \|_{\cal K}$ is commonly linked to the so-called notion of 'Gaussian width' \cite{widthbook}\cite{width} of set ${\cal K}$.

Note that $\| {\bf N} \|_{\cal K}$ is tightly linked with the quality of optimal approximation of ${\bf N}$ with an element of ${\cal K - K}$. An arithmetic relation can be checked directly in the form of
\begin{equation*}
\inf \limits_{{\bf Z} \in {\cal K - K}} \| {\bf N} - {\bf Z} \|_{F}^2 = \| {\bf N} \|_F^2 - \sup \limits_{{\bf V} \in {\cal K - K}, \| {\bf V} \|_F = 1} ({\bf N}, {\bf V})_F^2,
\end{equation*} 
which means that the smaller values of $\| {\bf N} \|_{\cal K}$ correspond to sets ${\cal K - K}$ that approximate the noise tensor ${\bf N}$ badly. Thus, the goal of the research can be seen as showing that the set of high-dimensional low-rank tensors is even {\it worse} at approximating white noise data, than the set of low-rank matrices.

However, studying {\it optimal} low-rank tensor approximations is highly challenging: even for certain particular tensors, like those arising in matrix multiplication complexity optimization, the {\it optimal} canonical approximations have not been found in decades. In this work, we will focus our theoretical study on the particular case of rank-one tensors (equivalent in any of considered definitions of 'rank')
\begin{equation*}
{\cal K} = \{ {\bf T} \in \mathbb{R}^{m_1 \times m_2 \times \dots \times m_d}: \mbox{rank}({\bf T}) \leq 1 \},
\end{equation*}
which is still nontrivial, since ${\cal K - K} = {\cal K + K}$ is then the set of tensors of canonical rank up to 2; the higher tensor ranks will then be studied numerically. It is clear that the 'equality' constraint in (\ref{eq:knorm}) can be replaced with the 'inequality' constraint of the form
\begin{equation}
\label{eq:2norm}
 \sup \limits_{\mbox{rank}({\bf V}) \leq 2, \|{\bf V}\|_F = 1} ({\bf N}, {\bf V})_F =  \sup \limits_{\mbox{rank}({\bf V}) \leq 2, \|{\bf V}\|_F \leq 1} ({\bf N}, {\bf V})_F.
\end{equation}
A seemingly appealing approach is utilizing the common Gaussian width properties, such as an explicit formula for the Gaussian width of a Minkowski sum of two bounded sets (which is equal to the sum of two respective Gaussian widths \cite{widthbook}), in order to reduce the rank bound down to simplest $1$ and work on a bound for
\begin{equation}
\label{eq:1norm}
\sup \limits_{\mbox{rank}({\bf V}) = 1, \|{\bf V}\|_F = 1} ({\bf N}, {\bf V})_F.
\end{equation}
However, we will not be able to follow that approach. While the bound for (\ref{eq:1norm}) is indeed {\it far} easier to establish as compared to (\ref{eq:2norm}), the set of $d \geq 3$ dimensional tensors
\begin{equation}
\label{eq:2set}
\{ {\bf V}: \mbox{rank}({\bf V}) \leq 2, \|{\bf V}\|_F \leq 1\}
\end{equation}
is just not a subset of a Minkowski sum of two identical sets
\begin{equation}
\label{eq:1set}
\{ {\bf V}: \mbox{rank}({\bf V}) = 1, \|{\bf V}\|_F \leq \Gamma \},
\end{equation}
no matter how large a constant $\Gamma > 0$ is, which follows from the non-closedness of the set of rank-$2$ canonical tensors. In other words, for arbitrary $\Gamma > 0$, a tensor ${\bf V}_{\Gamma}, \mbox{rank}({\bf V}_{\Gamma}) = 2, \| {\bf V}_{\Gamma} \|_F = 1$ exists such that it cannot be represented as a sum of two rank-one tensors with a norm smaller than $\Gamma$ \cite{silva}. 
	
Hence, the set of norm-bounded rank-two canonical tensors (\ref{eq:2set}) may not be represented as a sum of two identical norm-bounded sets of rank-one tensors, and we will not be able to use the common Gaussian width properties. The main technical section of the paper, that starts below, will provide a workaround for the necessity of considering rank-one components of exceedingly high norms.

\section{Bounds for the set of rank-two tensors}

Let us construct a bound for the particular value
\begin{equation}
\sup \limits_{\mbox{rank}({\bf V}) \leq 2, \|{\bf V}\|_F = 1} ({\bf N}, {\bf V})_F, \label{eq:target2}
\end{equation}	
where the term 'rank' is understood in the canonical sense. We will use the following extended notations for the canonical structure of ${\bf V}$:
\begin{align}
v & := \mbox{vec}({\bf V}), v = W \alpha =
\begin{bmatrix}
w_1 & w_2
\end{bmatrix} 
\alpha \label{eq:parrank2} \\
w_j & = \bigotimes_{s = 1}^d w_j^{(s)}, \| w_j^{(s)} \|_2 = 1, W^{(s)} = \begin{bmatrix}
w_1^{(s)} & w_2^{(s)}
\end{bmatrix}. \nonumber
\end{align}
In other words, we will assume that the columns of the canonical factor matrices $W^{(s)}$ are scaled to unit norm using additional coefficient vector $\alpha \in \mathbb{R}^2$.
Let us introduce some additional notations for the orthogonal bases of Kronecker product spans. Let ${\cal Q}_{2, \Omega}$ denote a set of orthogonal bases $Q \in \mathbb{R}^{\prod \limits_{s = 1}^d m_s \times 2}$, that correspond to at least one set of canonical factors $W^{(s)}$ with bounded condition numbers:

\begin{align}
W & = QR, Q^* Q = I_2, \label{eq:krorth} \\
W & = \begin{bmatrix}
w_1 & w_2
\end{bmatrix}, w_j = \bigotimes_{s = 1}^d w_j^{(s)}; \nonumber \\
W^{(s)} & = \begin{bmatrix}
w_1^{(s)} & w_2^{(s)}
\end{bmatrix}, \| w_j^{(s)} \|_2 = 1, \nonumber \\
\mbox{cond}_2 & (W) \leq \Omega. \label{eq:krcond}
\end{align} 

Let additionally ${\cal Q}_2 := \overline{{\cal Q}_{2, \infty}}$, where ${\cal Q}_{2, \infty}$ corresponds to the set of orthogonal bases from definition (\ref{eq:krorth}) without a condition number limit (\ref{eq:krcond}). As the set of rank-two canonical tensors is not closed in general, there can exist $\hat Q \in {\cal Q}_2$ such that $\hat Q \notin {\cal Q}_{2, \Omega}$ for any ${\Omega}$. Such an example is well-known for $d = 3, m_1 = m_2 = m_3 = 2$:
\begin{equation}
\hat Q = 
\begin{bmatrix}
1 & 0 \\
0 & \frac{1}{\sqrt{3}} \\
0 & 0 \\
0 & \frac{1}{\sqrt{3}} \\
0 & 0 \\
0 & 0 \\
0 & 0 \\
0 & \frac{1}{\sqrt{3}}
\end{bmatrix}. \label{eq:unboundcond}
\end{equation}

Now, we are ready to reformulate the target variable (\ref{eq:target2}) in terms of orthogonal projections with the following Lemma.
\begin{lemma}[Structured orthogonal projection]
	\label{lemma:widthtoq}
	\begin{equation*}
	\sup_{\mbox{rank}({\bf V}) \leq 2, \| {\bf V} \|_F = 1} ({\bf N}, {\bf V})_F = \max_{Q \in {\cal Q}_2} \| Q^* n \|_2, \mbox{where } n = \mbox{vec}({\bf N}).
	\end{equation*}
\end{lemma}
\begin{proof}
	Recall that every rank-2 tensor has the parametrizaton $\mbox{vec}({\bf V}) = W \alpha$, thus
	\begin{align*}
	({\bf N}, {\bf V})_F & = (n, W \alpha)_2 = (Q^* n, R \alpha)_2; \\
	\| {\bf V} \|_F & = \| QR \alpha \|_2 = \| R \alpha \|_2.
	\end{align*}
	As ${\bf V}$ is fully described by $Q$ and $R \alpha$, we obtain
	\begin{equation*}
	\sup_{\mbox{rank}({\bf V}) \leq 2, \| {\bf V} \|_F = 1} ({\bf N}, {\bf V})_F = \sup_{Q \in {\cal Q}_{2, \infty}} \max_{\| R \alpha \|_2 = 1} (Q^* n, R \alpha)_2 = \sup_{Q \in {\cal Q}_{2, \infty}} \| Q^* n \|_2 = \max_{Q \in {\cal Q}_2}  \| Q^* n \|_2.
	\end{equation*}
\end{proof}

Let us now annotate the main theoretical result of the paper.
\begin{theorem}[Gaussian width bound for the set of unit rank-two tensors]
\label{th:main}
	Assume ${\bf N} \in \mathbb{R}^{m \times m \times \dots \times m}$ is a $d$-dimensional tensor, each element of which is a random i.i.d. variable distributed according to standard normal distribution. There exists an absolute constant $\mu$, such that
	\begin{equation*}
	\mathbb{P}\{\sup_{\mbox{rank}({\bf V}) \leq 2, \| {\bf V} \|_F = 1} ({\bf N}, {\bf V})_F > \mu \sqrt{m d^2 \ln m} + t \} \leq e^{-\frac{t^2}{4}} + 2 e^{- \frac{m^{d / 2}}{8}}.
	\end{equation*}
\end{theorem}
Before diving into the details, let us first provide a plan/sketch of the proof of the Theorem \ref{th:main}. 
\begin{itemize}
	\item In order to directly utilize the fact that elements of ${\bf N}$ are i.i.d. Gaussian, we will aim to create a finite set $Q_1, Q_2, \dots Q_{\chi} \in {\cal Q}_{2, \infty}$, such that
	\begin{equation*}
	\max_{Q \in {\cal Q}_2}  \| Q^* n \|_2 \approx \max_{k = 1 \dots \chi}  \| Q_k^* n \|_2.
	\end{equation*}
	Then, since each $Q_k$ has orthogonal columns, that do not depend on $n$, we can utilize the fact that the two elements of the vector $Q_k^* n$ have a standard Gaussian distribution, and apply union bounds.
	\item In order to build a finite set $Q_1, Q_2, \dots Q_{\chi} \in {\cal Q}_{2, \infty}$, we are going to define $\epsilon$-grids over the unit spheres of the corresponding non-orthogonal factors $w_{1,2}^{(s)}$: the number of elements sufficient to cover a unit sphere with an $\epsilon$-grid is well-known. Since we are only interested in the orthogonal bases, the coefficients $\alpha$ can be omitted from consideration.
	\item Certain orthogonal bases $Q \in {\cal Q}_2$ can only be attained at a set of canonical factors $W, W^{(s)}$ with large/unbounded condition numbers (see example (\ref{eq:unboundcond}), \cite{silva}). These orthogonal bases are harder to approximate with an $\epsilon$-grid over the factors $w_{1,2}^{(s)}$, since $\epsilon$-perturbations of ill-conditioned factors may lead to arbitrarily large perturbations of the corresponding orthogonal subspace. In order to resolve the issue, a theorem is required that guarantees that considering $Q \in {\cal Q}_{2, \Omega}$ for some limited condition number bound $\Omega$ is sufficient for our purposes. 
\end{itemize}

Let us now begin the analysis with some technical lemmas. Firstly, let us provide a tool for handling cosines between perturbed pairs of vectors.
\begin{lemma}[Approximate cosine] \label{lemma:acosine}
Let $x, y, \hat x, \hat y \in \mathbb{R}^{m}$, and $\| x - \hat x \|_2 \leq \eta_x \| x \|_2, \| y - \hat y \|_2 \leq \eta_y \| y \|_2$. Then,
\begin{equation*}
\lvert \frac{(x, y)_2}{\| x \|_2 \| y \|_2} - \frac{(\hat x, \hat y)_2}{\| \hat x \|_2 \| \hat y \|_2} \rvert \leq 2 \eta_x + 2 \eta_y.
\end{equation*}
\end{lemma}
\begin{proof}
\begin{align*}
\lvert \frac{(x, y)_2}{\| x \|_2 \| y \|_2} - \frac{(\hat x, \hat y)_2}{\| \hat x \|_2 \| \hat y \|_2} \rvert & \leq \lvert \frac{(x, y)_2}{\| x \|_2 \| y \|_2} - \frac{(\hat x, y)_2}{\| \hat x \|_2 \| y \|_2} \rvert + \lvert \frac{(\hat x, y)_2}{\| \hat x \|_2 \| y \|_2} - \frac{(\hat x, \hat y)_2}{\| \hat x \|_2 \| \hat y \|_2} \rvert \\
& \leq \|\frac{x}{\| x \|} - \frac{\hat x}{\| \hat x \|} \|_2 + \|\frac{y}{\| y \|} - \frac{\hat y}{\| \hat y \|_2} \|_2.
\end{align*}
The two similar terms can then be bound by
\begin{align*}
\| \frac{x}{\| x \|_2} - \frac{\hat x}{\| \hat x \|_2} \|_2 & = \| \frac{x - \hat x}{\| x \|_2} +  \hat x ( \frac{1}{\| x \|_2} - \frac{1}{\| \hat x \|_2} ) \|_2 \\
& \leq \frac{\| x - \hat x \|_2}{\| x \|_2} + \frac{ \lvert \| \hat x \|_2 - \| x \|_2 \rvert }{ \| x \|_2} \leq 2 \eta_x; \\
\| \frac{y}{\| y \|_2} - \frac{\hat y}{\| \hat y \|_2} \|_2 & \leq \mbox{\{analogously\}} \leq 2 \eta_y.
\end{align*}
\end{proof}

Secondly, let us prove that the condition number of the full Kronecker-product base $W$ cannot exceed the condition number of per-dimension factor matrices $W^{(s)}$.
\begin{lemma} \label{lemma:cond}
For each dimension $s$, $\mbox{cond}_2(W^{(s)}) \geq \mbox{cond}_2(W)$.
\end{lemma}
\begin{proof}
Taking into account that $\|\ w_j^{(s)} \|_2 = 1$, it can be seen that both $W^* W$ and $(W^{(s)})^* W$ have a similar form of
\begin{equation*}
W^* W = \begin{bmatrix}
1 & w_1^* w_2 \\
w_2^* w_1 & 1 \\
\end{bmatrix}, 
(W^{(s)})^* W = \begin{bmatrix}
	1 & (w_1^{(s)})^* w_2^{(s)} \\
	(w_2^{(s)})^* w_1^{(s)} & 1 \\
\end{bmatrix}
\end{equation*}
It suffices to see that the eigenvalues of a matrix $\begin{bmatrix} 1 & \gamma \\ \bar \gamma & 1 \end{bmatrix}$ are expressed as $\lambda_{1,2} = 1 \pm \lvert \gamma \rvert$, and that 
\begin{equation*}
\lvert w_1^* w_2 \rvert = \lvert \prod_s (w_1^{(s)})^* w_2^{(s)} \rvert \leq \lvert (w_1^{(s)})^* w_2^{(s)} \rvert.
\end{equation*}
\end{proof}

Finally, let us provide a useful representation for an SVD of a two-column matrix with unit column norms.
\begin{lemma}
If $C = \begin{bmatrix} c_1 & c_2 \end{bmatrix} \in \mathbb{R}^{m \times 2}$ is a matrix with norm-one columns and 
\begin{equation*}
\| c_1 + c_2 \|_2 \geq \| c_1 - c_2 \|_2,
\end{equation*} 
then its short-form SVD is given by
\begin{equation*}
C = \begin{bmatrix} 
\frac{c_1 + c_2}{\| c_1 + c_2 \|_2} & \frac{c_1 - c_2}{\| c_1 - c_2 \|_2}
\end{bmatrix}
\begin{bmatrix} 
\frac{1}{\sqrt{2}}\| c_1 + c_2 \|_2 & 0 \\
0 & \frac{1}{\sqrt{2}}\| c_1 - c_2 \|_2
\end{bmatrix}
\begin{bmatrix}
\frac{1}{\sqrt{2}} & \frac{1}{\sqrt{2}} \\
\frac{1}{\sqrt{2}} & - \frac{1}{\sqrt{2}}
\end{bmatrix}.
\end{equation*}
\end{lemma}
\begin{proof}
The equality is verified by direct multiplication. It suffices to verify the orthogonality of the side matrices. To show orthogonality of the two columns of the first multiple, we can use
\begin{equation*}
(c_1 + c_2, c_1 - c_2) = \| c_1 \|_2^2 - \| c_2 \|_2^2 = 1 - 1 = 0.
\end{equation*}
\end{proof}
By scaling the provided SVD formula, a useful representation can be obtained:
\begin{equation*}
C = \frac{1}{2} \| c_1 + c_2 \|_2 \begin{bmatrix} 
\frac{c_1 + c_2}{\| c_1 + c_2 \|_2} & \frac{c_1 - c_2}{\| c_1 - c_2 \|_2}
\end{bmatrix}
\begin{bmatrix}
1 & 1 \\
\frac{\| c_1 - c_2 \|_2}{\| c_1 + c_2 \|_2} & -\frac{\| c_1 - c_2 \|_2}{\| c_1 + c_2 \|_2}\\
\end{bmatrix}
\end{equation*}
If we introduce additional notations $\alpha := \frac{\| c_1 - c_2 \|_2}{\| c_1 + c_2 \|_2}$, $a := \frac{c_1 + c_2}{\| c_1 + c_2 \|_2}, b := \frac{c_1 - c_2}{\| c_1 - c_2 \|_2}$, and recall that columns of $C$ are of unit norm, we can replace the scalar coefficient with individual column scaling, which gives a useful parametrization
\begin{align}
c_1 & = \frac{a + \alpha b}{\sqrt{1 + \alpha^2}}, \nonumber \\
c_2 & = \frac{a - \alpha b}{\sqrt{1 + \alpha^2}}, \label{eq:parrank2norm} \\
\mbox{cond}_2(C) & = \frac{1}{\alpha}. \nonumber
\end{align}

We are now ready to establish a theorem that guarantees that ${\cal Q}_{2, \Omega} \subset {\cal Q}_2$ is a rather dense inclusion: each element of ${\cal Q}_2$ can be approximated by an element of ${\cal Q}_{2, \Omega}$ with a uniform accuracy bound. Note that this result is the current bottleneck of the analysis: if similar result is established for higher tensor ranks, the rest of the analysis will follow. However, our proof method cannot be directly generalized towards higher ranks, since it utilizes explicit size-2 SVD decomposition formulas.
\begin{theorem}
\label{th:condappr}
For any $Q \in {\cal Q}_2$ and $\Omega \geq d$ there exists $\tilde Q \in {\cal Q}_{2, \Omega}$ such that
\begin{equation*}
\| Q Q^* - \tilde Q \tilde Q^* \|_2 \leq \delta_{\Omega},
\end{equation*}
where $\delta_{\Omega} = \frac{\mu_1}{\Omega^{\mu_2}}$, and $\mu_1, \mu_2 > 0$ are absolute constants (that do not depend on $Q$, $\Omega$, $d$).
\end{theorem}
\begin{proof}
As $Q \in {\cal Q}_2$, for any given $\varepsilon > 0$ there exists a $\hat Q \in {\cal Q}_{2, \infty}$, such that $\| QQ^* - \hat Q {\hat Q}^* \|_2 < \varepsilon$. Since $\varepsilon$ can be infinitely small, it suffices to prove the theorem for $Q \in {\cal Q}_{2, \infty}$. 

We will assume that the notation (\ref{eq:krorth}) holds. Assuming 
\begin{equation*}
\| w_1^{(s)} + w_2^{(s)}\|_2 > \| w_1^{(s)} - w_2^{(s)}\|_2, s = 1 \dots d
\end{equation*}
is not a loss of generality: since the orthogonal subspace spanned by columns of $Q$ does not depend on the signs of columns of $W = QR$, we can redefine $w_2^{(s)} \leftrightarrow - w_2^{(s)}$ if necessary. Then, we can utilize the parametrization (\ref{eq:parrank2norm}):
\begin{equation*}
w_1^{(s)} = \frac{a_s + \alpha_s b_s}{\sqrt{1 + \alpha_s^2}}, w_2^{(s)} = \frac{a_s - \alpha_s b_s}{\sqrt{1 + \alpha_s^2}}.
\end{equation*}

If at least one of the condition numbers of the factor bases $W^{(s)}$ is low enough: $\mbox{cond}_2(W^{(s)}) \leq \Omega$, then, by Lemma \ref{lemma:cond}, the bound $\mbox{cond}_2(W) \leq \Omega$ follows, thus $Q \in {\cal Q}_{2, \Omega}$ and the theorem holds trivially with $\tilde{Q} = Q$.

We can then assume that {\it all} factor bases $W^{(s)}$ are ill-conditioned: by (\ref{eq:parrank2norm}), $\alpha_s < \frac{1}{\Omega}, s = 1 \dots d$.
Let then by definition
\begin{align*}
 \tilde w_1^{(s)} & = \frac{a_s + \beta_s b_s}{\sqrt{1 + \beta_s^2}}, \tilde w_2^{(s)}  = \frac{a_s - \beta_s b_s}{\sqrt{1 + \beta_s^2}} \\
 \tilde w_1 & = \bigotimes_{s = 1} \tilde w_1^{(s)}, \tilde w_2 = \bigotimes_{s = 1} \tilde w_2^{(s)},
\end{align*}
where $\beta_s > \frac{1}{\Omega}$ are bounded from below to ensure that the condition number of $\tilde W = \begin{bmatrix} \tilde w_1 & \tilde w_2 \end{bmatrix}$ is limited.
Let us then show that the column subspaces of $W$ and $\tilde W$ are well aligned; in order to do so, let us consider the left singular bases of $W, \tilde W$ of the form
\begin{align*}
Q = \begin{bmatrix} q_1 & q_2 \end{bmatrix} & =
\begin{bmatrix} \frac{w_1 + w_2}{\| w_1 + w_2 \|_2} & \frac{w_1 - w_2}{\| w_1 - w_2 \|_2} \end{bmatrix} \\
\tilde Q = \begin{bmatrix} \tilde q_1 & \tilde q_2 \end{bmatrix} & =
\begin{bmatrix} \frac{\tilde w_1 + \tilde w_2}{\| \tilde w_1 + \tilde w_2 \|_2} & \frac{\tilde w_1 - \tilde w_2}{\| \tilde w_1 - \tilde w_2 \|_2} \end{bmatrix}
\end{align*}
Let us show that $q_1^* \tilde q_1, q_2^* \tilde q_2$ are both close to $1$, starting with the more complex $q_2^* \tilde q_2$. By opening the brackets in the Kronecker products,
\begin{align*}
w_1 - w_2 & = 2(\prod \limits_{s=1}^d \frac{1}{\sqrt{1 + \alpha_s}}) ( \alpha_1 b_1 \otimes a_2 \otimes \dots \otimes a_d + \alpha_2 a_1 \otimes b_2 \otimes \dots \otimes a_d + \dots \\ \dots & +  \alpha_d a_1 \otimes a_2 \otimes \dots \otimes b_d + e_{tail}), \\
\tilde w_1 - \tilde w_2 & = 2(\prod \limits_{s=1}^d \frac{1}{\sqrt{1 + \beta_s}}) ( \beta_1 b_1 \otimes a_2 \otimes \dots \otimes a_d + \beta_2 a_1 \otimes b_2 \otimes \dots \otimes a_d + \dots \\ \dots & +  \beta_d a_1 \otimes a_2 \otimes \dots \otimes b_d + \tilde e_{tail}),
\end{align*}
where
\begin{align*}
e_{tail} & = \sum \limits_{s_1} \sum \limits_{s_2 \neq s_1} \sum \limits_{s_3 \neq s_2, s_3 \neq s_1} \alpha_{s_1} \alpha_{s_2} \alpha_{s_3} z_{s_1 s_2 s_3} + \mbox{\{five nested sums\}} + \dots \\
\tilde e_{tail} & = \sum \limits_{s_1} \sum \limits_{s_2 \neq s_1} \sum \limits_{s_3 \neq s_2, s_3 \neq s_1} \beta_{s_1} \beta_{s_2} \beta_{s_3} z_{s_1 s_2 s_3} + \mbox{\{five nested sums\}} + \dots  \\
z_{s_1 s_2 \dots s_L} & := \bigotimes \limits_{s = 1}^d z_s, z_s = \begin{cases} a_s, s \neq s_1, s \neq s_2, \dots s \neq s_L;  \\
b_s, \mbox{otherwise}. \end{cases}
\end{align*}
Since all $z_{s_1 s_2 \dots s_L}$ constitute a set of $2^d$ orthogonal vectors, we can bound the tail norms using
\begin{align*}
\| \sum \limits_{s_1} \sum \limits_{s_2 \neq s_1} \sum \limits_{s_3 \neq s_2, s_3 \neq s_1} \alpha_{s_1} \alpha_{s_2} \alpha_{s_3} z_{s_1 s_2 s_3} \|_2^2 & = \sum \limits_{s_1} \sum \limits_{s_2 \neq s_1} \sum \limits_{s_3 \neq s_2, s_3 \neq s_1} \alpha_{s_1}^2 \alpha_{s_2}^2 \alpha_{s_3}^2 \\
& \leq \sum \limits_{s_1} \sum \limits_{s_2} \sum \limits_{s_3} \alpha_{s_1}^2 \alpha_{s_2}^2 \alpha_{s_3}^2 \\
& = \sum \limits_{s_1} \alpha_{s_1}^2 \sum \limits_{s_2} \alpha_{s_2}^2 \sum \limits_{s_3} \alpha_{s_3}^2 = (\| \alpha \|_2^2)^3;
\end{align*}
handling all the nested sums in the same way gives
\begin{align}
\| e_{tail} \|_2^2 & \leq (\| \alpha \|_2^2)^3 + (\| \alpha \|_2^2)^5 + (\| \alpha \|_2^2)^7 + \dots \leq \frac{(\| \alpha \|_2^2)^3}{1 - (\| \alpha \|_2^2)^2}; \nonumber \\
\| \tilde e_{tail} \|_2^2 & \leq (\| \beta \|_2^2)^3 + (\| \beta \|_2^2)^5 + (\| \beta \|_2^2)^7 + \dots \leq \frac{(\| \beta \|_2^2)^3}{1 - (\| \beta \|_2^2)^2}; \nonumber \\ 
\frac{\| e_{tail} \|_2^2 }{\| \alpha \|_2^2} & \leq \frac{(\| \alpha \|_2^2)^2}{1 - (\| \alpha \|_2^2)^2}; \frac{\| e_{tail} \|_2 }{\| \alpha \|_2} \leq \frac{\| \alpha \|_2^2}{\sqrt{1 - \| \alpha \|_2^4}} \leq 2 \| \alpha \|_2^2 \leq 2 \| \alpha\|_2; \label{eq:tailtoalpha} \\
\frac{\| \tilde e_{tail} \|_2^2 }{\| \beta \|_2^2} & \leq \frac{(\| \beta \|_2^2)^2}{1 - (\| \beta \|_2^2)^2}; \frac{\| \tilde e_{tail} \|_2 }{\| \beta \|_2} \leq \frac{\| \beta \|_2^2}{\sqrt{1 - \| \beta \|_2^4}} \leq 2 \| \beta \|_2^2 \leq 2 \| \beta \|_2; \nonumber
\end{align}
when $\| \alpha \|_2 < 0.9, \| \beta \|_2 < 0.9$. Since $q_2^* \tilde q_2$ corresponds to the cosine between $w_1 - w_2$ and $\tilde w_1 - \tilde w_2$, Lemma \ref{lemma:acosine} then gives
\begin{equation}\label{eq:cosdiff1}
\lvert (q_2, \tilde q_2)_2 - \frac{(\alpha, \beta)}{\| \alpha \|_2 \| \beta \|_2} \rvert \leq 4 \| \alpha \|_2 + 4 \| \beta \|_2.
\end{equation}
Let us now specify an appropriate $\beta$. The combined requirements for $\beta$ are the following:
\begin{equation*}
\begin{cases}
\| \beta \|_2 \leq 0.9 \\
\| \beta \|_C \geq \frac{1}{\Omega} \\
\frac{(\alpha, \beta)_2}{\| \alpha \|_2 \| \beta \|_2} \approx 1
\end{cases}.
\end{equation*}
The last two requirements can be fulfilled by a choice 
\begin{equation*}
\beta := \frac{1}{\Omega \| \alpha \|_C} \alpha;
\end{equation*}
Let us then bound the norm
\begin{equation*}
\| \beta \|_2 = \frac{\| \alpha \|_2}{\Omega \| \alpha \|_C} \leq \frac{\sqrt{d}}{\Omega} \leq \frac{1}{\sqrt{\Omega}}
\end{equation*}
under the theorem condition $\Omega \geq d$; since $d \geq 2$, $\|\beta \|_2 \leq \frac{1}{\sqrt{d}} < 0.9$. Since $\| \alpha \|_2~\leq~\| \beta \|_2$ by construction, combining the above bound with (\ref{eq:cosdiff1}) gives
\begin{align*}
(q_2, \tilde q_2)_2 \geq 1 - \frac{8}{\sqrt{\Omega}}.
\end{align*}
The bound for $(q_1, \tilde q_1)_2$ can be obtained in a similar way, but via considering even-sized nested sums; the main term of the expression for $w_1 + w_2$ is a vector
\begin{equation*}
w_1 + w_2 \approx 2(\prod \limits_{s=1}^d \frac{1}{\sqrt{1 + \alpha_s}}) (a_1 \otimes a_2 \otimes \dots \otimes a_d)
\end{equation*}
with a fixed unit norm. We will omit the analogous proof of
\begin{align*}
(q_1, \tilde q_1)_2 \geq 1 - \frac{8}{\sqrt{\Omega}}.
\end{align*}

Then, consider the $2 \times 2$ matrix $Q^* \tilde Q$; since the row and column norms of this product is limited by $1$, and the diagonal elements are equal to $1$ up to an error of order $\frac{1}{\sqrt{\Omega}}$, the offdiagonal elements of $Q^* \tilde Q$ are bounded by $\frac{\mu}{\Omega^{\frac{1}{4}}}$, where $\mu$ is an absolute constant. Thus, $Q^* \tilde Q$ is approximately equal to identity matrix with a residual bounded by $\frac{2 \mu}{\Omega^{\frac{1}{4}}}$ in Frobenius norm. Thus, the least singular value $\sigma_2(Q^* \tilde Q)$ is larger than $1 - \frac{2 \mu}{\Omega^{\frac{1}{4}}}$. Since $\sigma_2(Q^* \tilde Q)$ corresponds to the smallest cosine of the canonical angles between the subspaces $Q$ and  $\tilde Q$, the claim of the theorem follows from converting the bound towards the sines of these angles.
\end{proof}

In the next Lemma, let us prove that the set of orthogonal bases ${\cal Q}_{2, \Omega}$ can be well covered by a finite grid, and provide an estimate on the cardinality of the required set.
\begin{lemma}[Finite projector embeddings]
\label{lemma:embedding}
For any $1 \geq \zeta > 0, \Omega \geq d \geq 2$, there exists a finite set of bases $Q_1, Q_2, Q_3 \dots Q_{\chi_{\zeta}} \in {\cal Q}_{2, \infty}$, such that for each $Q \in {\cal Q}_{2, \Omega}$ there exists at least one $Q_j$ that suffices
\begin{equation}\label{eq:projdiff}
\| \tilde Q \tilde Q^* - Q_j Q_j^* \|_F \leq \zeta,
\end{equation}
and $\chi_{\zeta} \leq (\frac{\mu \Omega^2}{\zeta})^{2md}$, where $\mu$ is an absolute constant.
\begin{proof}
Recall that each $\tilde Q \in {\cal Q}_{2, \Omega}$ has a representation in the form (\ref{eq:krorth}). In order to constructively prove the statement of the Lemma, it is proposed to introduce finite grids over $2d$ unit spheres that contain the columns $\tilde w_1^{(s)}, \tilde w_2^{(s)}$. Since the condition numbers of $\tilde W^{(s)}$ are bounded, independent approximations for bases $\tilde W^{(s)}$, constructed using finite grids with appropriately scaled accuracy, will be sufficient to ensure (\ref{eq:projdiff}).

Assume that each vector $\tilde w_1^{(s)}, \tilde w_2^{(s)}$ (recall $\| \tilde w_1^{(s)} \|_2 = \| \tilde w_2^{(s)} \|_2 = 1$) is approximated using a finite grid over a unit sphere with accuracy $\tau = \frac{\zeta}{\Omega^2}$:
\begin{equation}\label{eq:grideps}
\| \tilde{w_1}^{(s)} - w_{1,j_{s,1}}^{(s)} \|_2 \leq \tau, \| \tilde{w_2}^{(s)} - w_{2,j_{s,2}}^{(s)} \|_2 \leq \tau.
\end{equation}

Assume $j$ is a multiindex over index combinations $j_{s,1}, j_{s,2}$, and define
\begin{equation*}
W_j := \bigotimes \limits_{s = 1}^d W^{(s)}_{j_s}, W^{(s)}_{j_s} := \begin{bmatrix} w_{1,j_{s,1}}^{(s)} & w_{2,j_{s,2}}^{(s)} \end{bmatrix}.
\end{equation*}

To proceed with the proof, let us establish two bounds on the approximation residuals $\| \Delta_b \|_F, \| \Delta_c \|_F$, where
\begin{align*}
\Delta_b & := \tilde W - W_j, \\
\Delta_c & := \tilde W^* \tilde W - W_j^* W_j.
\end{align*}
Using $\Omega \geq d \geq 2, \zeta \leq 1$, it can be seen that
\begin{align}
\| \tilde w_1 - w_{1, j_1} \|_2 & = \| \bigotimes_{s = 1}^d \tilde w_1^{(s)} - \bigotimes_{s = 1}^d (\tilde w_1^{(s)} + (w_{1, j_{s,1}}^{(s)} - \tilde w_1^{(s)})) \|_2 \leq (1 + \tau)^d - 1 \nonumber \\
& = (1 + \frac{\zeta}{\Omega^2})^d - 1 = \sum \limits_{k = 1}^d C^k_d (\frac{\zeta}{\Omega^2})^k = \sum \limits_{k = 1}^d \frac{C^k_d}{d^k} (\frac{\zeta}{\Omega})^k \leq \sum \limits_{k = 1}^d (\frac{\zeta}{\Omega})^k \nonumber \\
& \leq  \frac{\frac{\zeta}{\Omega}}{(1 - \frac{\zeta}{\Omega})} \leq \frac{\frac{\zeta}{\Omega}}{(1 - \frac{1}{d})} \leq 2 \frac{\zeta}{\Omega}, \label{eq:gridbound} \\
\| \tilde w_2 - w_{2, j_2} \|_2 & \leq \mbox{\{analogously\}} \leq 2 \frac{\zeta}{\Omega};
\| \Delta_b \|_F \leq 2\sqrt{2} \frac{\zeta}{\Omega}. \nonumber 
\end{align}
Since $\tilde w_1, \tilde w_2, w_{1, j_1}, w_{2, j_2}$ are all norm-one vectors, the four elements of $\Delta_c$ correspond to pairwise cosines between the columns of $\tilde W$ and $W_j$. Combining (\ref{eq:grideps}) with Lemma \ref{lemma:acosine} shows
\begin{equation*}
\| \Delta_c \|_F \leq 8 \frac{\zeta}{\Omega^2}.
\end{equation*}
Now let us transform $\| \tilde Q \tilde Q^* - Q_j Q_j^* \|_2$ into expressions that can be bounded using the proved bounds for $\| \Delta_b \|_F, \| \Delta_c \|_F$. By adding and subtracting $\tilde W (W_j^* W_j)^{-1} \tilde W^*$, it can be seen that
\begin{align*}
\| \tilde Q \tilde Q^* - Q_j Q_j^* \|_F & = \| \tilde W (\tilde W^* \tilde W)^{-1} \tilde W^* - W_j (W_j^* W_j)^{-1} W_j^* \|_F \\
& \leq \| \tilde W ((\tilde W^* \tilde W)^{-1} - (W_j^* W_j)^{-1}) \tilde W^* \|_F \\
& + \| \Delta_b (W_j^* W_j)^{-1} \tilde W^* \|_F + \| W_j (W_j^* W_j)^{-1} \Delta_b \|_F \\
& \leq \| \tilde W ((\tilde W^* \tilde W)^{-1} - (W_j^* W_j)^{-1}) \tilde W^* \|_F \\
& + 2 \| \Delta_b (W_j^* W_j)^{-1} \tilde W^* \|_F + \| \Delta_b (W_j^* W_j)^{-1} \Delta_b \|_F.
\end{align*}
Since $\| \tilde W \|_F^2 = 2 = \sigma_1^2(\tilde W) + \sigma_2^2(\tilde W)$ and $\sigma_1(\tilde W) = \mbox{cond}_2(\tilde W) \sigma_2(\tilde W)$, the second singular value can be bounded using $\Omega$ directly as
\begin{equation*}
\sigma_2^2(\tilde W) = \frac{2}{1 + \mbox{cond}_2(\tilde W)^2},
\end{equation*}
which then leads to
\begin{equation*}
\frac{1}{\sigma_2^2(\tilde W)} = \frac{1 + \mbox{cond}_2(\tilde W)^2}{2} \leq \frac{1 + \Omega^2}{2}.
\end{equation*} 
The three terms can then be bound individually.
\begin{align*}
\| \tilde W ((\tilde W^* \tilde W)^{-1} & - (\tilde W^* \tilde W - \Delta_c)^{-1}) \tilde W^* \|_F = \| \tilde W (\tilde W^* \tilde W)^{-1} \Delta_c (\tilde W^* \tilde W - \Delta_c)^{-1} \tilde W^* \|_F \\
& \leq \| \tilde W (\tilde W^* \tilde W)^{-1} \|_2 \| \Delta_c \|_F \| \sum \limits_{k = 0}^{\infty} ((W^* W)^{-1} \Delta_c)^k \|_2 \| (\tilde W^* \tilde W)^{-1} \tilde W^* \|_2 \\
& \leq \| \tilde W (\tilde W^* \tilde W)^{-1} \|_2 \| \Delta_c \|_F \frac{1}{1 - \| (\tilde W^* \tilde W)^{-1} \|_2 \| \Delta_c \|_2} \| (\tilde W^* \tilde W)^{-1} \tilde W^* \|_2 \\
& = \frac{1}{\sigma_2^2(\tilde W)} \| \Delta_c \|_F \frac{1}{1 - \frac{\| \Delta_c \|_2}{\sigma_2^2(\tilde W)}} \leq 4 \zeta \frac{\Omega^2 + 1}{\Omega^2} \frac{1}{1 - 4 \zeta \frac{\Omega^2 + 1}{\Omega^2}} \leq \frac{8 \zeta}{1 - 8 \zeta}; \\
\| \Delta_b (W_j^* W_j)^{-1} \tilde W^* \|_F & = \| \Delta_b (\tilde W^* \tilde W - \Delta_c)^{-1} \tilde W^* \|_F \\
& = \| \Delta_b \sum \limits_{k = 0}^{\infty} ((W^* W)^{-1} \Delta_c)^k (\tilde W^* \tilde W)^{-1} \tilde W^* \|_F \\
& \leq \| \Delta_b \|_F \| \sum \limits_{k = 0}^{\infty} ((W^* W)^{-1} \Delta_c)^k \|_2 \| (\tilde W^* \tilde W)^{-1} \tilde W^* \|_2 \\ 
& \leq \frac{\| \Delta_b \|_F}{\sigma_2(\tilde W)} \frac{1}{1 - \frac{\| \Delta_c \|_2}{\sigma_2^2(\tilde W)}} \leq 2\sqrt{2} \zeta \frac{\sqrt{1 + \Omega^2}}{\Omega}  \frac{1}{1 - 4 \zeta \frac{\Omega^2 + 1}{\Omega^2}} \leq \frac{4 \zeta}{1 - 8 \zeta}; \\
\| \Delta_b (W_j^* W_j)^{-1} \Delta_b \|_F & = \| \Delta_b (\tilde W^* \tilde W - \Delta_c)^{-1} \Delta_b \|_F \\
& \leq \| \Delta_b \|^2_F \| (\tilde W^* \tilde W)^{-1} \|_2 \| \sum \limits_{k = 0}^{\infty} ((W^* W)^{-1} \Delta_c)^k \|_2 \\
& \leq \frac{\| \Delta_b \|_F^2}{\sigma_2^2(\tilde W)} \frac{1}{1 - \frac{\| \Delta_c \|_2}{\sigma_2^2(\tilde W)}} \leq 8 \zeta \frac{1 + \Omega^2}{\Omega^2}  \frac{1}{1 - 4 \zeta \frac{\Omega^2 + 1}{\Omega^2}} \leq \frac{16 \zeta}{1 - 8 \zeta},
\end{align*}
where the Neumann sum expression $(I - A)^{-1} = \sum \limits_{k = 0}^{\infty} A^k$ that were utilized multiple times are valid under the assumption $\zeta < \frac{1}{8}$. With a more strict assumption $\zeta < \frac{1}{16}$ combining the above bounds shows
\begin{equation*}
\| \tilde Q \tilde Q^* - Q_j Q_j^* \|_F \leq \zeta \frac{32 \zeta}{1 - 8 \zeta} < 64 \zeta.
\end{equation*}
Recall $j$ is a multiindex over combinations of $2d$ indeces $j_{s, 1}, j_{s, 2}$. The range of values $j_{s, 1}, j_{s, 2}$ corresponds to the number of elements on a unit sphere, sufficient to form a finite $\tau$-grid, which can be bounded by $(\frac{6}{\tau})^m = (\frac{6 \Omega^2}{\zeta})^m$ each (the sufficient cardinalities of $\epsilon$-grids for unit spheres can be found in \cite{widthbook}, for example). Thus, the value $\chi$, which corresponds to the range of the multiindex $j$, can be bounded by the products of ranges of indeces $j_{s, 1}, j_{s, 2}$, which gives
\begin{equation*}
\chi = (\frac{6 \Omega^2}{\zeta})^{2md}.
\end{equation*}
Rescaling $\zeta \rightarrow \hat \zeta = \frac{\zeta}{64} < \frac{1}{16}, \tau \rightarrow \hat \tau = \frac{\hat \zeta}{\Omega^2}$ finishes the proof.
\end{proof}
\end{lemma}

We are now ready to provide a proof for the Theorem \ref{th:main}. Recall the claim:
\begin{equation*}
\mathbb{P}\{\sup_{\mbox{rank}({\bf V}) \leq 2, \| {\bf V} \|_F = 1} ({\bf N}, {\bf V})_F > \mu \sqrt{m d^2 \ln m} + t \} \leq e^{-\frac{t^2}{4}} + 2 e^{- \frac{m^{d / 2}}{8}}
\end{equation*}
for a constant $\mu$ and random i.i.d. standard Gaussian elements of a $d$-dimensional tensor ${\bf N} \in \mathbb{R}^{m \times m \times \dots \times m}$.
\begin{proof}
Let $Q \in {\cal Q}_2$ and $n := vec({\bf N})$. Assume $Q_1, Q_2 \dots Q_{\chi}$ is a finite projector embedding, that fulfills the terms of Lemma \ref{lemma:embedding}. Assume the condition number $\Omega$ and the approximation quality $\zeta$ to be parameters that will be defined later, and assume that $\tilde Q \in {\cal Q}_{2, \Omega}$ approximates the subspace of $Q$ as in Theorem \ref{th:condappr}. Then, for each $j$
\begin{equation*}
\| Q^* n \|_2 = \| Q Q^* n \|_2 \leq \|Q_j Q_j^* n \|_2 + \| Q_j Q_j^* - \tilde Q \tilde Q^* \|_2 \| n \|_2 + \| \tilde Q \tilde Q^* - Q Q^* \|_2 \| n \|_2.
\end{equation*}
Since $\tilde Q \in {\cal Q}_{2, \Omega}$, by Lemma \ref{lemma:embedding} there exists at least one $j$, such that \mbox{$\| Q_j Q_j^* - \tilde Q \tilde Q^* \|_2 \leq \zeta$}; for that index $j$ it then holds
\begin{equation*}
\| Q^* n \|_2 \leq \| Q_j^* n \|_2 + (\delta + \zeta) \| n \|_2,
\end{equation*}
where the following conditions on the parameters $\delta, \zeta$ are fulfilled ($\mu_k$ are absolute constants):
\begin{align}
\delta & = \frac{\mu_1}{\Omega^{\mu_2}}; \mbox{ (Theorem \ref{th:condappr})} \label{eq:deltadef} \\
\chi & = (\frac{\mu_3 \Omega^2}{\zeta})^{2md}. \mbox{ (Lemma \ref{lemma:embedding})} \label{eq:chidef}
\end{align}
As the bounds of Theorem \ref{th:condappr} and Lemma \ref{lemma:embedding} have no dependence on $Q$, optimizing over $Q \in {\cal Q}_2$ gives
\begin{equation}
\label{eq:balance}
\max_{Q \in {\cal Q}_2} \| Q^* n \|_2 \leq \max_{j = 1 \dots \chi} \| Q_j^* n \|_2 + (\delta + \zeta) \| n \|_2
\end{equation}
Since the selection of the finite embedding $Q_1, Q_2 \dots Q_{\chi}$ does not depend on $n$, for each particular $j$ each element of the vector $Q_j^* n \in \mathbb{R}^2$ is a random variable with a standard normal distribution. Then, since \mbox{$\|Q_j n \|_2 \leq \sqrt{2} \|Q_j n\|_C$}, and $\max \limits_{j = 1 \dots \chi} \|Q_j n\|_C$ is a maximum over $2 \chi$ dependent standard normal random variables, a union bound can be applied:
\begin{align}
\mathbb{P}\{ \max \limits_{j = 1 \dots \chi} \| Q_j^* n \|_2 > 2 \sqrt{\ln \chi} + t\} & \leq 4 \chi \mathbb{P} \{{\cal N}(0,1) > \sqrt{2 \ln \chi} + \frac{t}{\sqrt{2}}\} \nonumber \\
& \leq \chi e^{-\frac{(\sqrt{2 \ln \chi} + \frac{t}{\sqrt{2}})^2}{2}} \label{eq:unionbound} \\
& < \chi e^{-\frac{(\sqrt{2 \ln \chi})^2}{2}} e^{-\frac{t^2}{4}} = e^{-\frac{t^2}{4}}, \nonumber 
\end{align}
where a common Gaussian tail bound 
\begin{equation*}
\mathbb{P} \{{\cal N}(0,1) > t \} \leq \frac{1}{t\sqrt{2\pi}}e^{-\frac{t^2}{2}}
\end{equation*}
is applied, and an irrelevant term $\frac{4}{\sqrt{4\pi \ln \chi}} < 1$ is discarded.

The bound (\ref{eq:unionbound}) suggests that $\max \limits_{j = 1 \dots \chi} \| Q_j^* n \|_2 \lesssim 2 \sqrt{\ln \chi}$. Let us then balance the terms of (\ref{eq:balance}) to the same value of ${\cal G}$:
\begin{align}
2 \sqrt{\ln \chi} & \leq \frac{\cal G}{3}; \nonumber \\
\delta \| n \|_2 & \leq \frac{\cal G}{3}; \label{eq:balancecrude} \\
\zeta \| n \|_2 & \leq \frac{\cal G}{3}. \nonumber
\end{align}
The value $\| n \|_2^2$ has a chi-squared distribution with a high number $m^d$ of degrees of freedom, and is tightly concentrated around its mean: by employing the Laurent-Massart bound (\cite{laurentmassart}, section 4.1), one can see that
\begin{equation}
\mathbb{P} \{ \lvert \| n \|_2^2 - m^d \rvert \leq \frac{m^d}{2} \} \geq \mathbb{P} \{ \lvert \| n \|_2^2 - m^d \rvert \leq \frac{m^d}{4} + \frac{m^{\frac{d}{2}}}{4} \} \geq 1 - 2 e^{-\frac{m^{\frac{d}{2}}}{8}} \label{eq:chi2bound}
\end{equation}
If the concentration event (\ref{eq:chi2bound}) holds, the sufficient conditions for the balancing (\ref{eq:balancecrude}) can be seen as 
\begin{align}
	2 \sqrt{\ln \chi} & \leq \frac{\cal G}{3}; \label{eq:balancechi} \\
	\delta & \leq \frac{\cal G}{4 m^{\frac{d}{2}}}; \label{eq:balancedelta} \\
	\zeta & \leq \frac{\cal G}{4 m^{\frac{d}{2}}}. \label{eq:balancezeta}
	\end{align}
Let us show that the balancing can be carried out using a particular choice of parameters. Recalling (\ref{eq:deltadef}), (\ref{eq:balancedelta}) can be fulfilled with equality by specifying $\Omega$:
\begin{align*}
\frac{\mu_1}{\Omega^{\mu_2}} & = \frac{\cal G}{4 m^{\frac{d}{2}}}; \Omega = \frac{\mu_4 m^{\frac{d}{2 \mu_2}}}{{\cal G}^{\frac{1}{\mu_2}}}, \mbox{ where } \mu_4 := (4 \mu_1)^{\frac{1}{\mu_2}}.
\end{align*}
Plugging the value of $\Omega$ into (\ref{eq:chidef}), we obtain
\begin{equation*}
\chi = (\frac{\mu_5 m^{\frac{d}{\mu_2}}}{\zeta {\cal G}^{\frac{2}{\mu_2}}})^{2md}, \mu_5 := \mu_3 \mu_4^2;
\end{equation*}
Similarly to $\mu_4, \mu_5$, the further appearing absolute constants $\mu_6, \mu_7 \dots$ are directly expressible using the previous ones, but we will stop providing explicit formulas for readability. By applying a logarithm, we obtain
\begin{equation}
\ln \chi = \mu_6 m d + \mu_7 m d^2 \ln m - \mu_8 m d \ln \zeta - \mu_9 m d \ln {\cal G}. \label{eq:almost}
\end{equation}
Taking (\ref{eq:balancezeta}) with equality as a definition for $\zeta$, 
\begin{equation*}
\ln \zeta = \ln {\cal G} - \mu_{10} - \mu_{11} d \ln m.
\end{equation*}
Plugging it into (\ref{eq:almost}) and using (\ref{eq:balancechi}) gives a final inequality on ${\cal G}$:
\begin{equation}
\label{eq:mainineq}
{\cal G}^2 \geq \mu_{12} md + \mu_{13} md^2 \ln m - \mu_{14} md \ln {\cal G}.
\end{equation}
The selection
\begin{equation*}
{\cal G}^2 := \mu_{15} m d^2 \ln m, {\cal G} = \mu_{16} \sqrt{m d^2 \ln m}
\end{equation*}
guarantees that the inequality is fulfilled (it can be checked that $\mu_{13} > 1$, thus $\ln \mu_{16} > 0$, and the last term of (\ref{eq:mainineq}) is negative), which concludes the proof.
\end{proof}

\section{Numerical experiments}
In this section, we will provide a set of numerical experiments on synthetic data, in order to
\begin{itemize}
\item Check the sharpness of the Theorem \ref{th:main} bounds;
\item Numerically study the tensor approximation perturbations 
\begin{equation*}
\varepsilon_{{\cal P_K}}({\bf T}, {\bf N}) := \| {\cal P_K}({\bf T} + {\bf N}) - {\bf T} \|_F.
\end{equation*}
for higher-rank tensors. Since for $R > 1$, the canonical, Tensor-Train and Tucker low-rank tensor representations are non-equivalent, the obtained numerical asymptotic estimates will be different for each representation.
\end{itemize}

Let us first numerically check the rank-one bounds of Theorem \ref{th:main}. Let us fix total tensor element count $M := m^d$ at $M = 2^{24}$, and consider variable tensor dimensionalities $d$ with combinations \mbox{$\{d = 2, m = 2^{12}\}$}, \mbox{$\{d = 3, m = 2^{8}\}$} \dots \mbox{$\{d = 24, m = 2\}$}. Assume for each possible $d$ a rank-one tensor is generated factor-wise randomly, and scaled to unit Frobenius norm. Then, an element-wise i.i.d. Gaussian noise is applied to the tensor, and the ALS \cite{als} is applied to the result as ${\cal P_K}$ in order to numerically assess the perturbation variable \mbox{$\varepsilon_{{\cal P_K}}({\bf T}, {\bf N})$} defined as in (\ref{eq:objective}). 

The results are provided on Figure \ref{fig:dimension}. The 'Projection error' graph corresponds to the experimental values of \mbox{$\varepsilon_{{\cal P_K}}({\bf T}, {\bf N})$}, while the 'ALS residual' and the 'Noise norm' lines depict the left and right sides of the hypothesis (\ref{eq:hypo}). Note that in all cases these two values are {\it very} close. When a high dimensionality $d$ is combined with an exceedingly high value of the noise norm (Figure \ref{fig:dimensionhigh}), the hypothesis (\ref{eq:hypo}) may fail (by a tiny margin), which commonly results in a complete loss of approximation \mbox{$\varepsilon_{{\cal P_K}}({\bf T}, {\bf N})$}. That numerically confirms the necessity and reasonableness of the proposed hypothesis, and shows the imperfection of the ALS algorithm.

Note that in practice the issue can be resolved by successive 'multi-grid' approximations: if for a small $d = 3$ the ALS satisfies the hypothesis (\ref{eq:hypo}) and returns a decent approximation \mbox{$\varepsilon_{{\cal P_K}}({\bf T}, {\bf N})$}, then the resulting factors of size $2^{8}$ can be organized into matrices of size \mbox{$2^4 \times 2^4$}, that can be SVD-approximated with a rank of one. The resulting $2 \times 3 = 6$ factors of size $2^4$ can be fed as input to another run of ALS with dimension $d = 6$. The factors of the six-dimensional approximation can be in a similar way recursively be used to construct an initial approximation for an ALS with $d = 12$, and so on. Empirically, such a strategy allows better high-dimensional approximations in case of data subject to large noises.

The results of the experiment on the Figure \ref{fig:dimension} are well-approximated by an asymptotic formula, presented as a yellow line. Let us compare this asymptotic formula with a scaled theoretical bound of Theorem \ref{th:main}, when both are expressed with $d, M$:

\begin{align*}
\varepsilon_{{\cal P_K}}({\bf T}, {\bf N}) & \approx \sqrt{\frac{dM^{\frac{1}{d}}}{M}} \mathbb{E} \| {\bf N} \|_F & \mbox{\{empirical, yellow line on Figure \ref{fig:dimension}\}}; \\
\varepsilon_{{\cal P_K}}({\bf T}, {\bf N}) & \lesssim \sqrt{\frac{d M^{\frac{1}{d}} \log(M)}{M}} \mathbb{E} \| {\bf N} \|_F, w.h.p. & \mbox{\{scaled Theorem \ref{th:main}\}}.
\end{align*}
The results indicate that the proved Theorem \ref{th:main} is only inaccurate up to a minor term $\log(M)$ and a constant, and reflects the overall asymptotic behavior well.

\begin{figure}
	\subfloat[Low noise norm $\|{\bf N}\|_F = \frac{1}{10} \|{\bf T}\|_F$]
	{\includegraphics[width=0.45\textwidth]{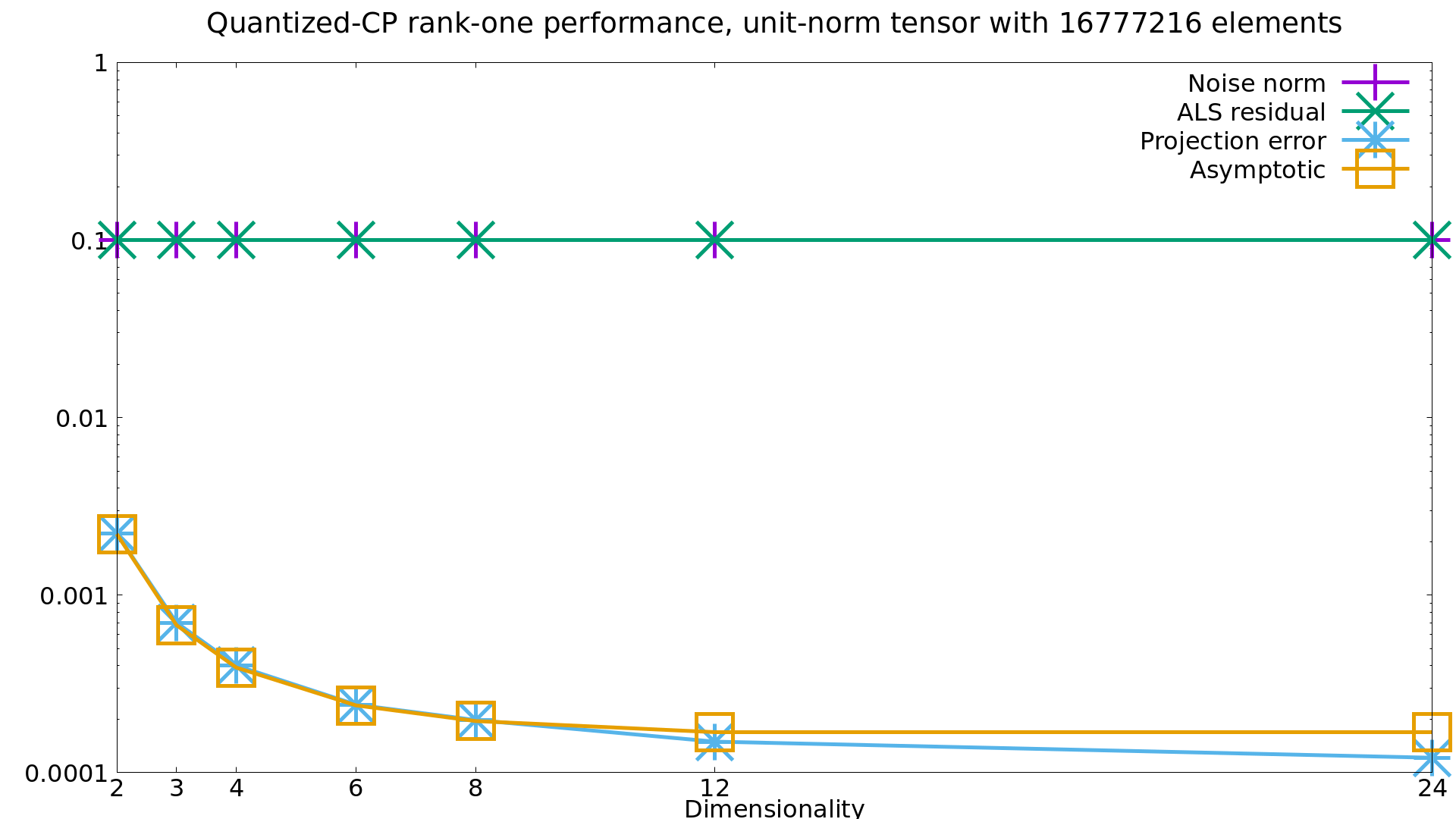}
	\label{fig:dimensionlow}}
	\hfill
	\subfloat[High noise norm $\|{\bf N}\|_F = 10 \|{\bf T}\|_F$]
	{\includegraphics[width=0.45\textwidth]{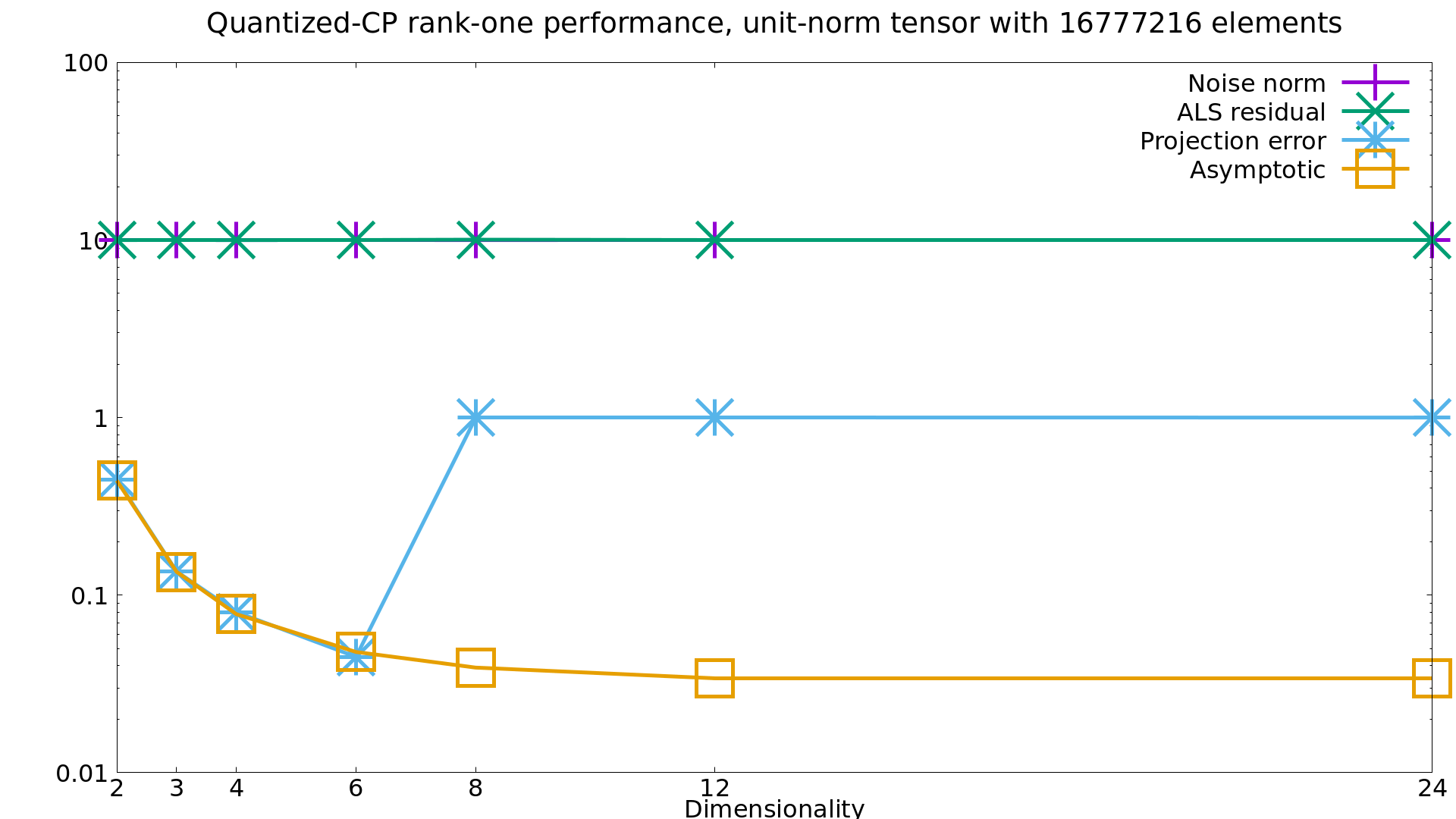} \label{fig:dimensionhigh}}

	\caption{Rank-one approximation perturbations in dependence of dimensionality.}
	\label{fig:dimension}	
\end{figure}

Let us now provide a numerical study of the case of higher ranks. In case $R > 1$ the variables $\varepsilon_{{\cal P_K}}({\bf T}, {\bf N})$ and $\| N \|_{F, {\cal K}}$ depend on the type of a low-rank tensor representation. The following formats were considered:
\begin{itemize}
\item Canonical tensor format; ${\cal P_K}$ then denotes the result of the ALS \cite{als} algorithm.
\item Tensor-Train format; ${\cal P_K}$ then denotes the result of the TT-SVD \cite{tt} algorithm, run with {\it each} tensor-train rank limited to the same value of $R$;
\item Tucker tensor format; ${\cal P_K}$ then denotes the result of the HOSVD \cite{hosvd} algorithm, run with {\it each} Tucker rank limited to the same value of $R$.
\end{itemize}

The experiments were carried out as follows. For each of the three structures, a sequence of tensors $\{ {\bf T_r}, \mbox{rank}({\bf T_r}) = r\}$ with a same size and norm, but varied rank were generated randomly, subjected to an element-wise random i.i.d. Gaussian noise, and then approximated by the corresponding algorithm ${\cal P_K}$. The numerically obtained values $\varepsilon_{{\cal P_K}}({\bf T_r}, {\bf N})$ were then fit to a curve of the form $C r^{\alpha}$ by solving a least-squares problem in a logarithmic scale, and the resulting $\alpha$ values were of most interest. The random tensor generation details include
\begin{itemize}
	\item Canonical tensor format: $M = 2^{18}$, each canonical factor elements of ${\bf T_r}$ was chosen as a i.i.d. Gaussian variables, and the resulting tensor was scaled to unit Frobenius norm;
	\item Tensor-Train format: $M = 2^{18}$, ${\bf T_r}$ were selected as a TT-SVD approximations to a element-wise random i.i.d. Gaussian tensors, scaled to unit Frobenius norm;
	\item Tucker tensor format: $M = 2^{20}$, the Tucker core elements were selected as i.i.d. Gaussian variables, and the Tucker factors were selected as orthogonal bases of subspaces spanning random i.i.d. Gaussian columns. The resulting tensor product was then also scaled to a unit Frobenius norm.
\end{itemize}

The obtained results are provided on Figures \ref{fig:canonical}, \ref{fig:tt}, \ref{fig:tucker}. It can be see that $\alpha$ does not change with dimensionality $d$ for the canonical and Tensor-Train formats, but does depend on $d$ for the Tucker format. Empirically, under the assumption (\ref{eq:hypo}) on the approximation algorithm ${\cal A}_{\cal K}$ it holds
\begin{equation}
\label{eq:mainrr}
\varepsilon_{{\cal A}_{\cal K}}({\bf T}, {\bf N}) \lesssim
\left[
\begin{array}{l}
\sqrt{\frac{d R M^{\frac{1}{d}}\log d}{M}} \mathbb{E} \| {\bf N} \|_F, \mbox{ for a canonical rank}; \\
\sqrt{\frac{d R^2 M^{\frac{1}{d}}\log d}{M}} \mathbb{E} \| {\bf N} \|_F, \mbox{ for a tensor train rank}; \\
\sqrt{\frac{d R^d M^{\frac{1}{d}}\log d}{M}} \mathbb{E} \| {\bf N} \|_F, \mbox{ for Tucker rank}.
\end{array}
\right .
\end{equation}

 The obtained canonical and Tucker format empirical bounds directly correspond to each other, since a rank-$R$ Tucker tensor always has a rank-$R^d$ canonical representation, but the Tensor-Train bound should, apparently, be studied independently.
 
 It is noticeable that for all considered representations the obtained bound~(\ref{eq:mainrr}) for the approximation quality $\varepsilon_{{\cal P}_{\cal K}}({\bf T}, {\bf N})$ has a same asymptotic dependence with the number of representation parameters.

\begin{figure}
	\subfloat[Tensor dimension $d$ set to $3$]
	{\includegraphics[width=0.45\textwidth]{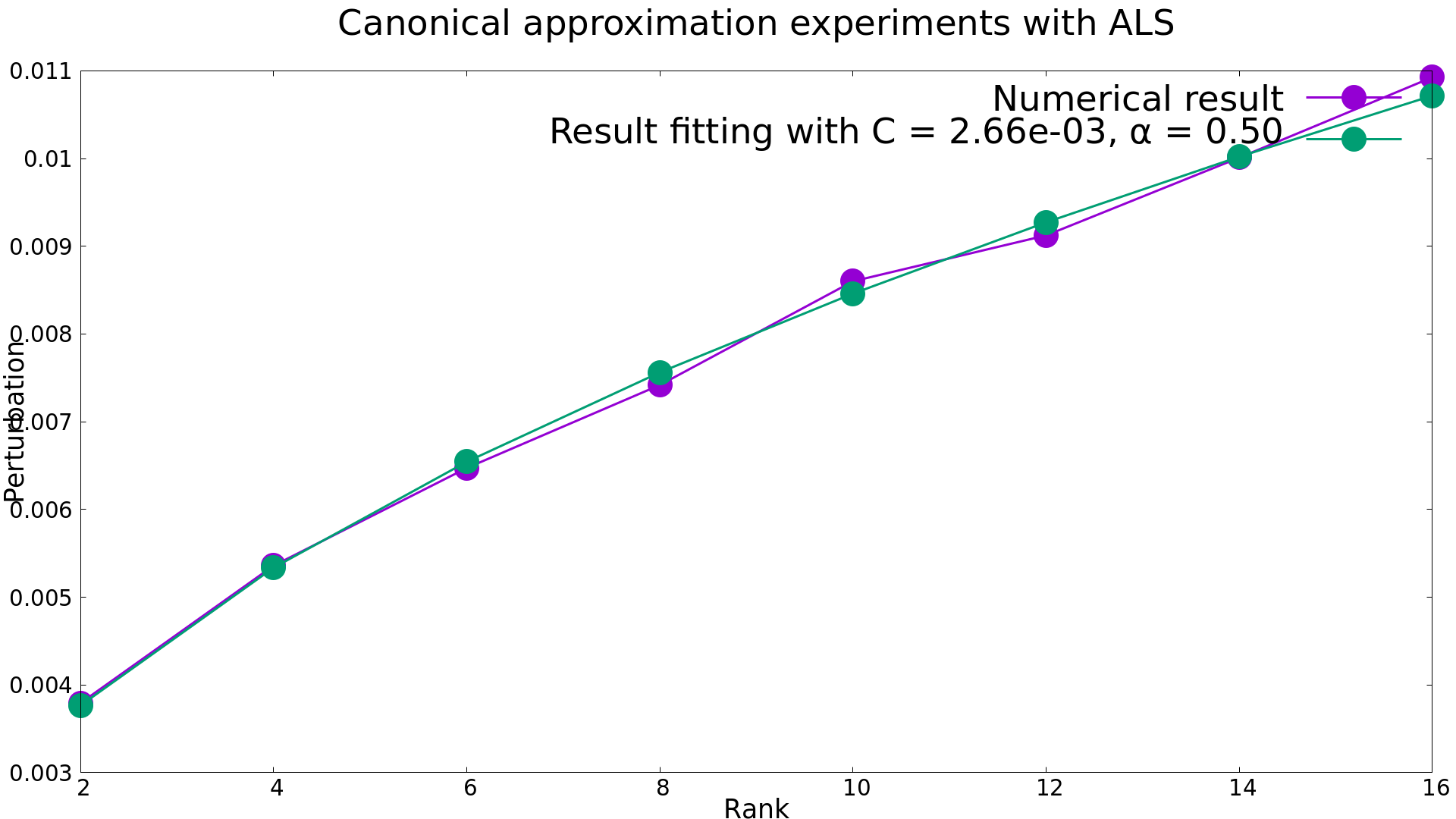}}
	\hfill
	\subfloat[Tensor dimension $d$ set to $6$]
	{\includegraphics[width=0.45\textwidth]{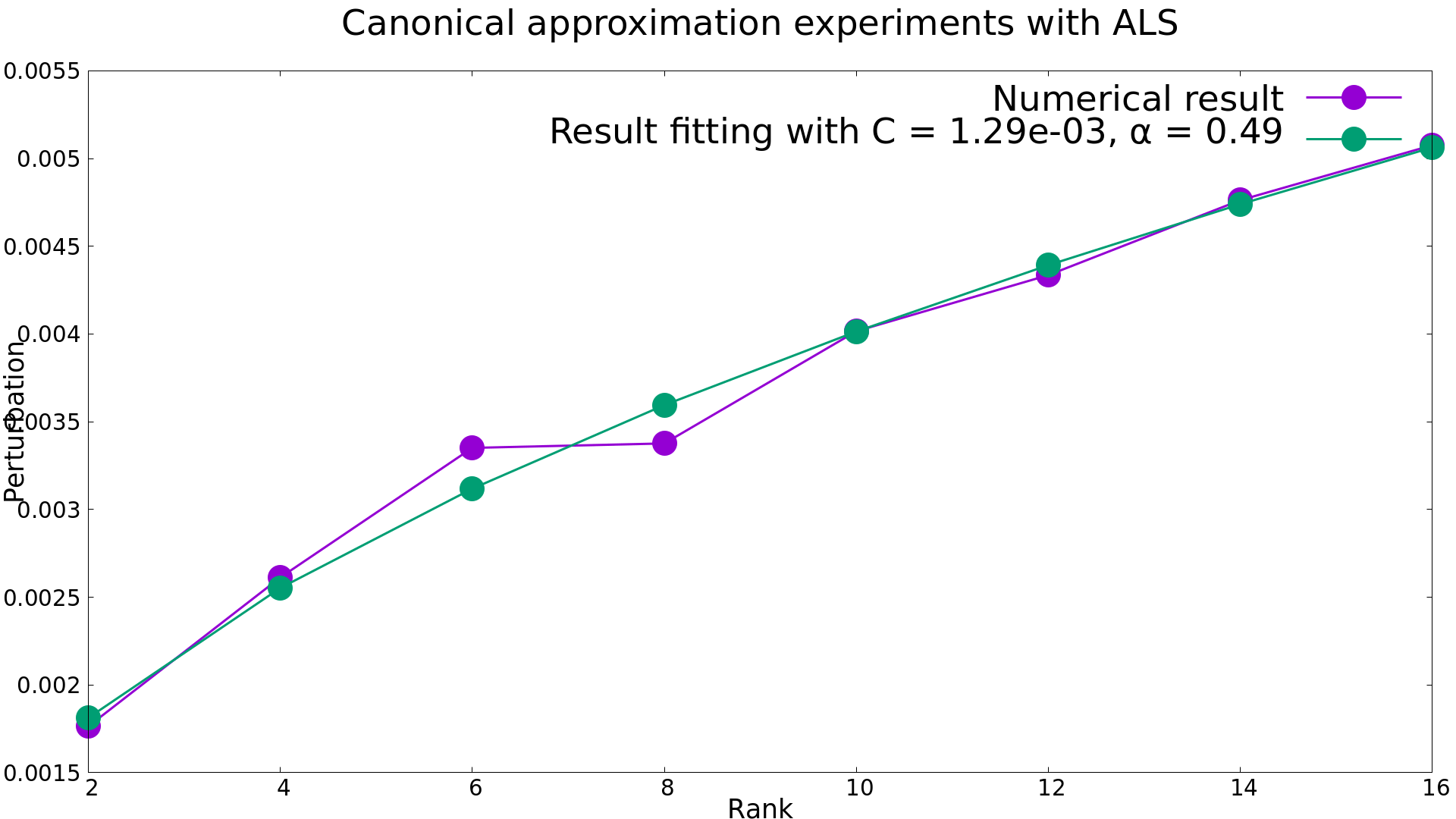}}
	
	\caption{Canonical ALS approximation perturbation dependence on the tensor rank}
	\label{fig:canonical}	
\end{figure}

\begin{figure}
	\subfloat[Tensor dimension $d$ set to $3$]
	{\includegraphics[width=0.45\textwidth]{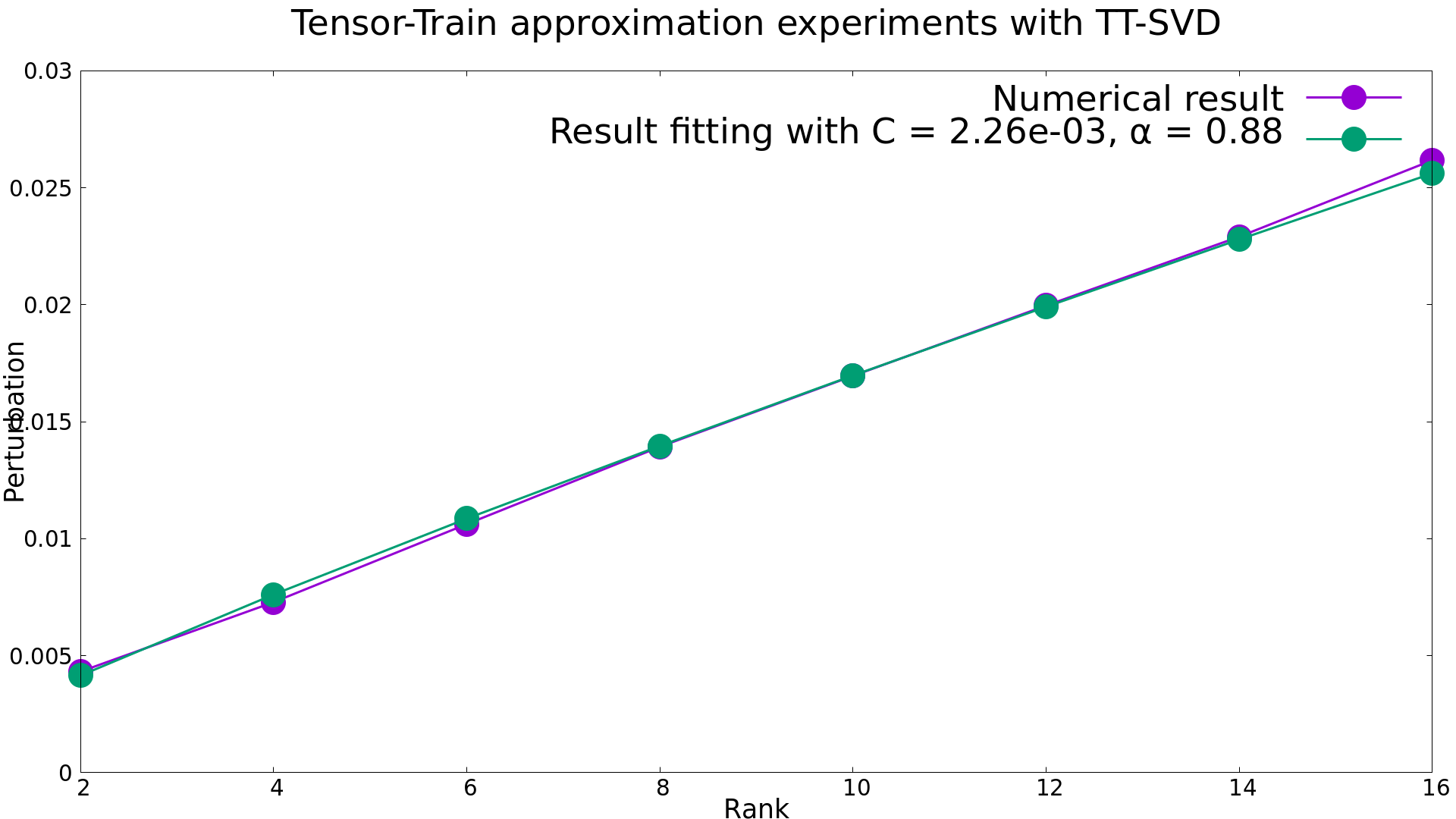}}
	\hfill
	\subfloat[Tensor dimension $d$ set to $6$]
		{\includegraphics[width=0.45\textwidth]{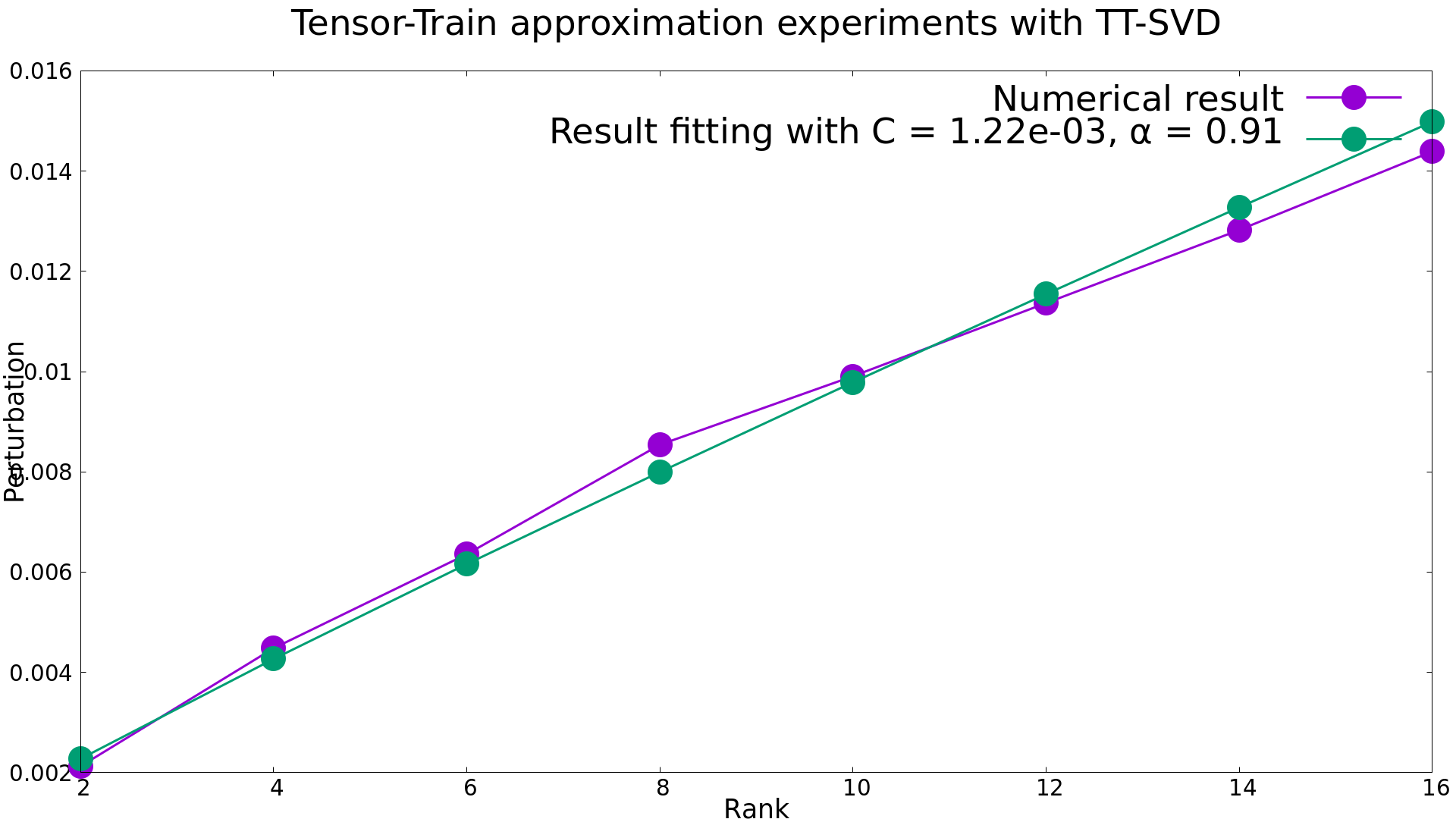}}
		
		\caption{Tensor-Train TT-SVD approximation perturbation dependence on the tensor rank}
		\label{fig:tt}	
	\end{figure}

\begin{figure}
	\subfloat[Tensor dimension $d$ set to $4$]
	{\includegraphics[width=0.45\textwidth]{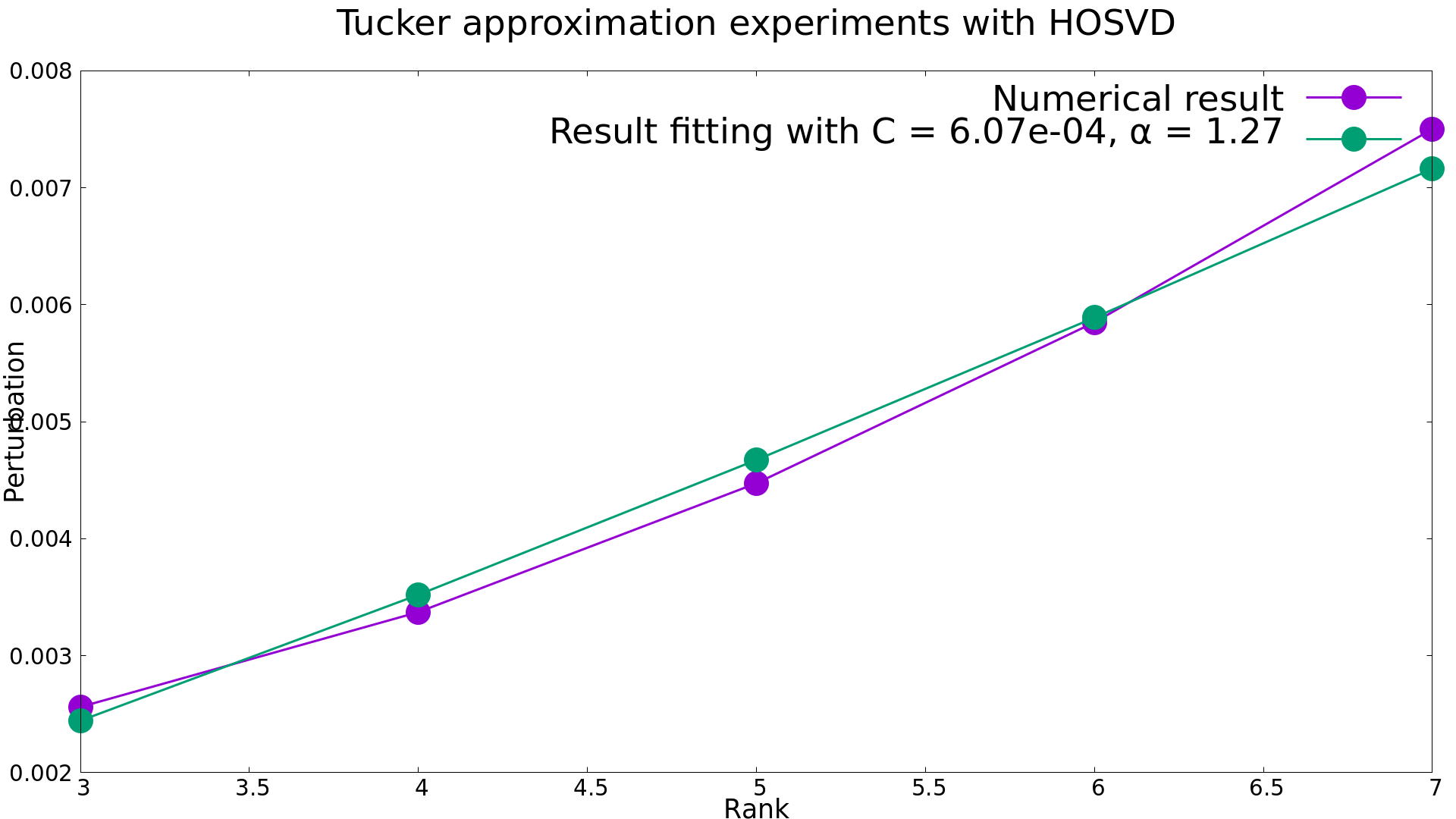}}
	\hfill
	\subfloat[Tensor dimension $d$ set to $5$]
	{\includegraphics[width=0.45\textwidth]{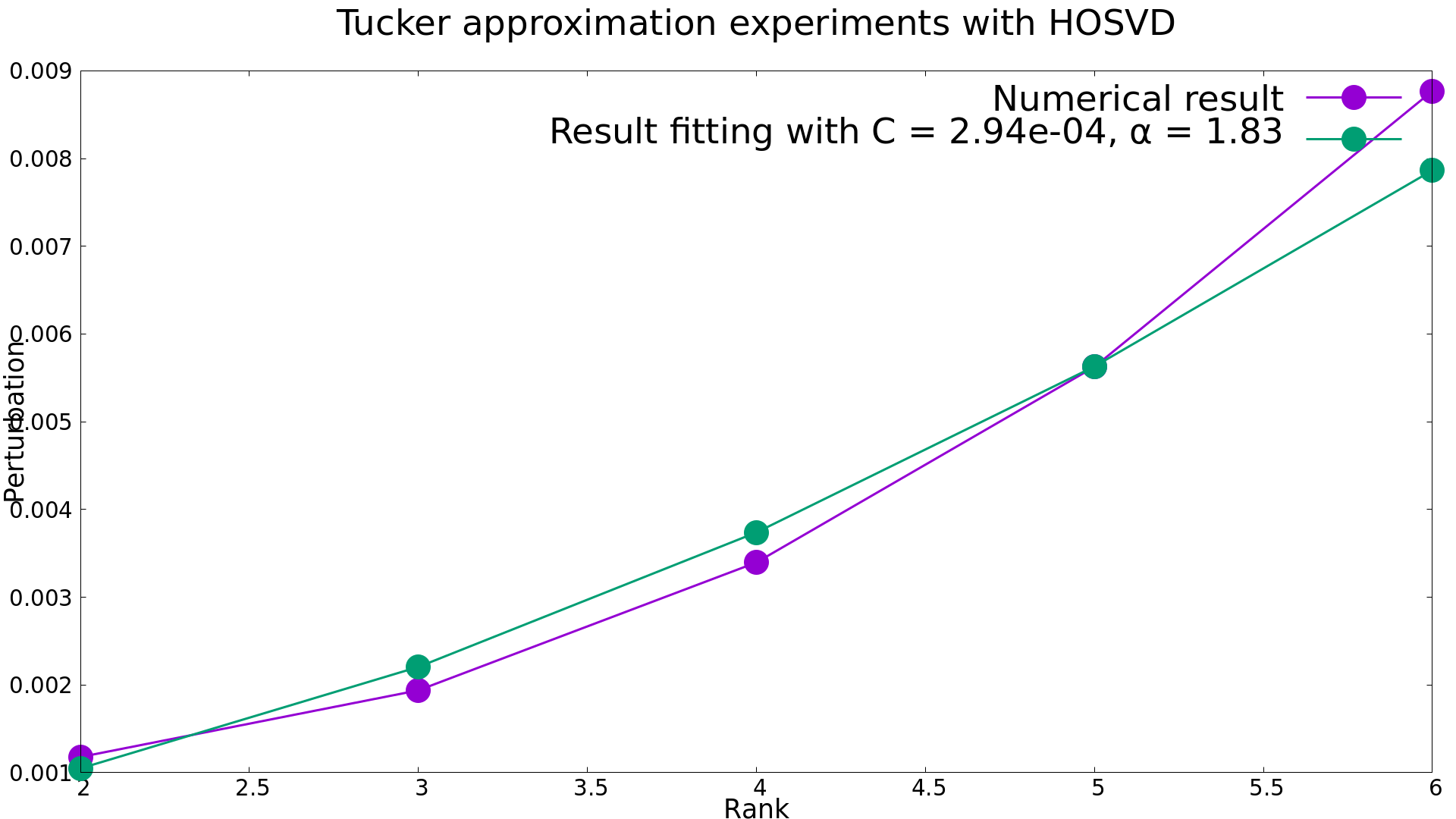}}
	
	\caption{Tucker HOSVD approximation perturbation dependence on the tensor rank}
	\label{fig:tucker}	
\end{figure}

\section{Discussion}

In this work, a novel theoretical bound is provided, which describes the improvement of the low-rank tensor noise filtration properties with the growth of structure dimensionality. Furthermore, a comparative numerical analysis is provided for the filtering properties of the canonical, Tensor-Train and Tucker low-rank representations.

The provided analysis has direct implications to the algorithms in the field of wireless communications, where the physical structure of the signal corresponds to a product of the so-called steering vectors of the form \cite{almeida}

\begin{equation*}
h(\phi) =
\begin{bmatrix}
1 & e^{i \phi} & e^{i 2 \phi} & e^{i 3 \phi} ... \\
\end{bmatrix}^T.
\end{equation*}

The exponent properties imply that if $h(\phi) \in \mathbb{C}^{M}$, then for any integer decomposition $M = \prod \limits_{s=1}^d m_s$ a tensorized ${\bf H} \in \mathbb{C}^{m_1 \times m_2 \times \dots \times m_d}$ such that \mbox{$h(\phi) = \mbox{vec}({\bf H})$} has an exact rank-one representation. It then follows that the use of {\it artificial} dimensions $m_1, m_2 \dots m_d$ in tensor algorithms, also known as {\it quantization}, may greatly improve the noise-filtering properties of tensors in wireless applications.

\backmatter

\bmhead{Contributions}
S.P. provided the main theoretical result derivation, numerical experiments and the first manuscript draft. N.Z. initially proposed the idea of utilizing Gaussian width framework, provided supervision, proof simplifications, and draft revisions. All authors have read and approved the final manuscript. 





\bibliography{sn-bibliography}


\end{document}